\documentclass[10pt,oneside]{article}
\usepackage{mathrsfs}
\usepackage{amsmath,amsfonts,amssymb,amsthm}
\usepackage{caption}
\usepackage{subfigure}
\usepackage{cite}
\usepackage{color}
\usepackage{graphicx}
\usepackage{subfigure}
\usepackage{float}
\usepackage{paralist}
\usepackage{enumerate}
\usepackage{dsfont}
\usepackage{bm}
\usepackage[colorlinks,
linkcolor=red,
citecolor=blue
]{hyperref}
\usepackage{authblk}
\usepackage[T1]{fontenc}
\allowdisplaybreaks

\newtheorem{thm}{Theorem}[section]
\newtheorem{lem}[thm]{Lemma}
\newtheorem{prop}[thm]{Proposition}
\newtheorem{rem}{Remark}[section]
\def\QEDopen{{\setlength{\fboxsep}{0pt}\setlength{\fboxrule}{0.2pt}\fbox{\rule[0pt]{0pt}{1.3ex}\rule[0pt]{1.3ex}{0pt}}}} %
\def\QED{\QEDopen} %
\def\endproof{\hspace*{\fill}~\QED\par\endtrivlist\unskip}%
\def\bma#1\ema{{\allowdisplaybreaks\begin{split}#1\end{split}}}

\numberwithin{equation}{section}

\begin{document}
	\title{Global well-posedness and large-time behavior of classical solutions to the Euler-Navier-Stokes system in $\mathbb{R}^3$} 
	\author[a,b]{ Feimin Huang   \thanks{E-mail: fhuang@amt.ac.cn(F.-M. Huang)}}
	\author[a]{ Houzhi Tang   \thanks{E-mail: houzhitang@amss.ac.cn(H.-Z. Tang)}}
	\author[c]{ Guochun Wu   \thanks{E-mail: guochunwu@126.com(G.-C. Wu)}}
	\author[d]{Weiyuan Zou \thanks{E-mail: zwy@amss.ac.cn(W.-Y. Zou)}}
\affil[a]{Academy of Mathematics and Systems Science, Chinese Academy of Sciences, Beijing 100190, P. R. China}
	\affil[b]{School  of  Mathematical  Sciences,  University  of  Chinese  Academy  of  Sciences,  Beijing 100049, P. R. China }
	\affil[c]{Fujian Province University Key Laboratory of Computational Science, School of Mathematical Sciences, Huaqiao University, Quanzhou 362021, P. R. China} 
	\affil[d]{College of Mathematics and Physics,
		Beijing University of Chemical Technology, Beijing 100029, P. R. China} 
	\date{}
	\renewcommand*{\Affilfont}{\small\it}	
	\maketitle
	
	
	\begin{abstract}
		In this paper, we study the Cauchy problem of a two-phase flow system consisting of the compressible isothermal Euler equations and the incompressible Navier-Stokes equations coupled through the drag force, which can be formally derived from the Vlasov-Fokker-Planck/incompressible Navier-Stokes equations. When the initial data is a small perturbation around an equilibrium state, we prove the global well-posedness of the classical solutions to this system and show the solutions tends to the equilibrium state as time goes to infinity. In order to resolve the main difficulty arising from the pressure term of the incompressible Navier-Stokes equations, we properly use the Hodge decomposition, spectral analysis, and energy method to obtain the $L^2$ time decay rates of the solution when the initial perturbation belongs to $L^1$ space. Furthermore, we show that the above time decay rates are optimal.
	\end{abstract}
		\textbf{2020 MR Subject Classification:}\ 35B40, 35B65, 76N10.\\
	\section{Introduction}
	\hspace{2em} In this paper we are concerned with the global well-posedness and large-time behavior of a coupled hydrodynamic system in three-dimensional space as follows,
	\begin{equation}\label{Main1}
		\left\{\begin{array}{llll}
			\displaystyle \partial_t\rho+\textrm{div}(\rho u)=0,\\
			\displaystyle \partial_t(\rho u)+\mathrm{div}\big(\rho u\otimes u\big)
			+\nabla\rho=-\rho(u-v),\\
			\displaystyle \partial_tv+v\cdot\nabla v+\nabla P=\Delta v+\rho(u-v),\\
			\displaystyle \text{div}v=0,
		\end{array}\right.
	\end{equation}
where $\rho=\rho(t,x)$ and $u=u(t,x)$ are the density and velocity for the compressible Euler fluid flow, $v=v(t,x)$ is the velocity for the incompressible Navier-Stokes fluid flow, respectively.

 We supply \eqref{Main1} with the initial data
\begin{equation}\label{ID1}
(\rho,u,v)|_{t=0}=(\rho_{0},u_{0},v_{0}),
\end{equation}
and the far-field states 
\begin{equation}\label{ID2}
\lim_{|x|\rightarrow +\infty}(\rho,u,v)=(\rho_*,0,0),
\end{equation}
where $\rho_*>0$ is the given positive constant.

	This coupled Euler-Navier-Stokes (E-NS) system can be formally derived from the Vlasov-Fokker-Planck/incompressible Navier-Stokes equations, which describe the behavior of a large cloud of particles interacting with the incompressible fluid in the following form:
	\begin{equation}\label{NSVFP}
		\left\{\begin{array}{llll}
			\displaystyle \partial_tf+\xi\cdot\nabla_xf+\text{div}_\xi((v-\xi)f)=-\alpha\text{div}_\xi((u_f-\xi)f)+\sigma\Delta_\xi f,\\
			\displaystyle  \partial_tv+v\cdot\nabla_x v+\nabla_x P=\Delta_x v+\int_{\mathbb{R}^3}(\xi-v)fd\xi,\\
			\displaystyle \text{div}_xv=0,
		\end{array}\right.
	\end{equation}
	for $(t,x,\xi)\in \mathbb{R}_+\times \mathbb{R}^3\times \mathbb{R}^3$, where  $f=f(t,x,\xi)$ denotes the distribution of particles, $v=v(t,x)$ is the velocity of incompressible fluid, and $u_f$ represents the averaged local velocity defined by 
	\begin{equation}\label{uf}
		u_f=u_f(t,x)=\frac{\int_{\mathbb{R}^3}\xi f(t,x,\xi)d\xi}{\int_{\mathbb{R}^3}f(t,x,\xi)d\xi}.
	\end{equation}
	
	Recently, this type of coupled kinetic-fluid model has received  a bulk of attention due to its wide range of applications in the modeling of reaction flows of sprays, atmospheric pollution modeling,  chemical engineering or waste water
	treatment, dust collecting units \cite{MR2226800,MR2041452,1981Collective,2006Large,Williams1958Spray,MR709743}. There have been many important mathematical works. In particular,  Carrillo, Choi, and Karper \cite{MR3465376} studied global existence, hydrodynamic limit, and large-time behavior of weak solutions to the system \eqref{NSVFP} via energy method and relative entropy techniques. For other interesting works, we refer to \cite{MR2106334,MR2106333,MR2729436,MR3073216,MR3227296,MR3403400,MR4076066,MR4420295}.
		
	Now we will carry out a formal derivation of the main system \eqref{Main1} by asymptotic analysis. We take into account a regime with $\alpha=\sigma=\varepsilon^{-1}$. Let $(f^\varepsilon,v^\varepsilon, P^\varepsilon)$ be the corresponding solution, that is 
	\begin{equation}\label{NSVFP1}
		\left\{\begin{array}{llll}
			\displaystyle \partial_tf^\varepsilon+\xi\cdot\nabla_xf^\varepsilon+\text{div}_\xi((v-\xi)f^\varepsilon)=-\frac{1}{\varepsilon}\text{div}_\xi((u_{f^\varepsilon}-\xi)f^\varepsilon)+\frac{1}{\varepsilon}\Delta_\xi {f^\varepsilon},\\
			\displaystyle  \partial_tv^\varepsilon+v^\varepsilon\cdot\nabla_xv^\varepsilon+\nabla_x P^\varepsilon=\Delta_x v^\varepsilon+\int_{\mathbb{R}^3}(\xi-v^\varepsilon)f^\varepsilon d\xi,\\
			\displaystyle \text{div}_xv^\varepsilon=0.
		\end{array}\right.
	\end{equation}
	In terms of $\eqref{NSVFP1}_1$, we formally have that
	$$
	-\text{div}_\xi((u_{f^\varepsilon}-\xi)f^\varepsilon)+\Delta_\xi f^\varepsilon \rightarrow 0, \quad \text{as}\quad \varepsilon \rightarrow 0.
	$$
	Thus, the particle distribution function $f^\varepsilon(t,x,\xi)$ converges to 
	\begin{equation}\label{ft}
		f(t,x,\xi)=\frac{\rho_f(t,x)}{(2\pi)^{3/2}}e^{-\frac{|u_f-\xi|^2}{2}},\quad \text{and}\quad \rho_{f}(t,x)=\int_{\mathbb{R}^3}f(t,x,\xi)d\xi.
	\end{equation}
	Integrating $\eqref{NSVFP1}_1$ with respect to $\xi$ over $\mathbb{R}^3$ and assuming the limits $f^\varepsilon  \rightarrow f$ and $u_{f^\varepsilon} \rightarrow u_f$ hold~$\text{as}$ ~$\varepsilon \rightarrow 0$, we can obtain the continuity equation
	\begin{equation}\label{mass}
		\partial_t\rho_f+\text{div}_x(\rho_fu_f)=0.
	\end{equation}
	Multiplying $\eqref{NSVFP1}_2$ by $\xi$ and integrating the equation with respect to $\xi$ in $\mathbb{R}^3$ yield 
	\begin{equation}\label{TR1}
		\begin{aligned}
			&\frac{d}{dt}\int_{\mathbb{R}^3} \xi f^\varepsilon(t,x,\xi)d\xi\\
			&=\int_{\mathbb{R}^3}\xi\Big(-\xi\cdot\nabla_xf^\varepsilon-\text{div}_\xi((v^\varepsilon-\xi)f^\varepsilon)\Big)d\xi
			+\int_{\mathbb{R}^3}\xi \Big(-\frac{1}{\varepsilon}\text{div}_\xi((u_{f^\varepsilon}-\xi)f^\varepsilon)+\frac{1}{\varepsilon}\Delta_\xi f^\varepsilon\Big)d\xi\\
			&=-\text{div}_x\Big(\int_{\mathbb{R}^3}\xi\otimes \xi f^\varepsilon d\xi\Big)+\int_{\mathbb{R}^3}(v^\varepsilon-\xi)f^\varepsilon d\xi\\
			&\triangleq I_1^\varepsilon+I_2^\varepsilon,	  		
		\end{aligned}
	\end{equation}
	where we use the fact that 
	$$
	\int_{\mathbb{R}^3}(u_{f^\varepsilon}-\xi)f^\varepsilon d\xi=0.
	$$
	The first term  on the right-hand side of \eqref{TR1} is stated as 
	\begin{equation}
		\begin{aligned}
			I_1^\varepsilon&=-\text{div}_x\Big(\int_{\mathbb{R}^3}(\xi-u_{f^\varepsilon})\otimes (\xi-u_{f^\varepsilon})f^\varepsilon d\xi+2\int_{\mathbb{R}^3}\xi\otimes u_{f^\varepsilon}f^\varepsilon d\xi-\int_{\mathbb{R}^3}u_{f^\varepsilon}\otimes u_{f^\varepsilon}f^\varepsilon d\xi \Big).
		\end{aligned}
	\end{equation}
	By \eqref{uf}, the second term $I_2^\varepsilon$ is equivalent to 
	$$
	I_2^\varepsilon=-\rho_{f^\varepsilon}(u_{f^\varepsilon}-v^\varepsilon).
	$$
A direct computation implies that 	
	\begin{equation}\label{f2}
		\int_{\mathbb{R}^3}(\xi-u_{f^\varepsilon})\otimes (\xi-u_{f^\varepsilon})
		\frac{1}{(2\pi)^{3/2}}e^{-\frac{|u_f-\xi|^2}{2}}d\xi={\rm{I}},\quad 
		\int_{\mathbb{R}^3}\frac{1}{(2\pi)^{3/2}}e^{-\frac{|u_f-\xi|^2}{2}}d\xi=1,
	\end{equation}
	where ${\rm{I}}$ denotes the identity matrix. In virtue of \eqref{uf}, \eqref{f2}, and assuming $v^\varepsilon \rightarrow v, P^\varepsilon\rightarrow P$ as $\varepsilon\rightarrow 0$, we obtain 
	\begin{equation}
		\lim_{\varepsilon \rightarrow 0} \frac{d}{dt}\int_{\mathbb{R}^3}\xi f^\varepsilon d\xi=\frac{d}{dt} \int_{\mathbb{R}^3}\xi fd\xi=\frac{d}{dt}(\rho_fu_f),
	\end{equation}
	
	\begin{equation}
		\begin{aligned}
			I_1
			&=\lim_{\varepsilon \rightarrow 0} I_1^\varepsilon=-\text{div}_x\Big(
			\int_{\mathbb{R}^3}(\xi-u_f)\otimes (\xi-u_f)\frac{\rho_f}{(2\pi)^{3/2}}e^{-\frac{|u_f-\xi|^2}{2}}d\xi\\
			&\quad+2\int_{\mathbb{R}^3}\xi\otimes u_f f d\xi-\int_{\mathbb{R}^3}u_f\otimes u_f\frac{\rho_f}{(2\pi)^{3/2}}e^{-\frac{|u_f-\xi|^2}{2}}d\xi\Big) \\
			&=-\text{div}_x(\rho_f{\rm{I}}+2\rho_fu_f\otimes u_f-\rho_fu_f\otimes u_f)\\
			&=-\nabla_x\rho_f-\text{div}_x(\rho_fu_f\otimes u_f),
		\end{aligned}
	\end{equation}
	and 
	\begin{equation}\label{I2}
		I_2=\lim_{\varepsilon \rightarrow 0}I_2^\varepsilon=-\rho_f(u_f-v).
	\end{equation}
	Thus, the momentum equation can be derived as 
	\begin{equation}\label{mm}
		\partial_t(\rho_fu_f)+\text{div}_x(\rho_fu_f\otimes u_f)+\nabla_x \rho_{f}=-\rho_f(u_f-v).
	\end{equation}
	From $\eqref{NSVFP1}_{2,3}$, \eqref{I2}, \eqref{mass}, and \eqref{mm}, we have  
	\begin{equation}
		\left\{\begin{array}{llll}
			\displaystyle \partial_t\rho_f+\text{div}_x(\rho_f u_f)=0,\\
			\displaystyle  \partial_t(\rho_fu_f)+\text{div}_x(\rho_fu_f\otimes u_f)+\nabla \rho_{f}=-\rho_f(u_f-v),\\
			\displaystyle \partial_tv+v\cdot\nabla_x v+\nabla_x P=\Delta_x v+\rho_f(u_f-v),\\
			\displaystyle \text{div}v=0,
		\end{array}\right.
	\end{equation}
	which is \eqref{Main1} by setting $\rho=\rho_f,~u=u_f,~\nabla=\nabla_x,~\text{div}=\text{div}_x,~\Delta=\Delta_x$.
	
	When the pressure term $\nabla\rho$ in $\eqref{Main1}_2$ vanishes, the coupled system is reduced to the pressureless Euler type system. Choi and Jung \cite{MR4316126} firstly applied the weighted energy method to investigate the global well-posedness and proved the solutions tend to the equilibrium state at the almost optimal decay rates. Later, the optimal decay rates were achieved in \cite{huang2023,Choi23,Choi20232}. However, to the best of our knowledge, there is no any work on the global well-posedness of the system \eqref{Main1}-\eqref{ID2} in $\mathbb{R}^3$. In this paper, we focus on the system \eqref{Main1}-\eqref{ID2}.
	
	
	To overcome the difficulty for the lack of the dissipation of velocity $u(t,x)$, we define a new variable $a=\ln \rho$.  Then the system \eqref{Main1}-\eqref{ID2} can be reformulated as follows:
	\begin{equation}\label{Main2}
		\left\{
		\begin{aligned}
			&\partial_ta+u\cdot\nabla a +\text{div}u=0,\\
			&\partial_tu+u\cdot\nabla u+\nabla a+u-v=0,\\
			&\partial_tv+v\cdot\nabla v+\nabla P=\Delta v+e^{a}(u-v),\\
			&\text{div} v=0,
		\end{aligned}
		\right.
	\end{equation}
	with the initial data
	\begin{equation}
		(a,u,v)|_{t=0}=(a_0,u_0,v_0),~a_0\triangleq\ln\rho_0,
	\end{equation}
	and far-field states
	\begin{equation}\label{in-data}
		(a,u,v)\rightarrow (a_*,0,0),~a_*\triangleq \ln\rho_* \quad \text{as}\quad |x|\rightarrow+\infty.
	\end{equation}

The first theorem on the global well-posedness of classical solutions to the Cauchy problem \eqref{Main2}-\eqref{in-data} is given as:
	\begin{thm}\label{thm1}
	Assume that for some integer $s\geq 3$, the initial data $(a_0-a_*,u_0,v_0)\in H^s(\mathbb{R}^3)$  satisfies
	\begin{equation}\label{in-data-e0}
		\|(a_0-a_*,u_{0},v_0)\|_{H^s(\mathbb{R}^3)}\leq \varepsilon_0,
	\end{equation}
	where  $\varepsilon_0$ is a small positive constant and ${\rm{div}}v_0=0$, then the Cauchy problem \eqref{Main2}-\eqref{in-data} admits a unique global classical solution $(a,u,v)$ such that 
	\begin{equation}\label{main-est}
		\|(a-a_*,u,v)(t)\|_{H^s(\mathbb{R}^3)}^2+\int_0^t\big( \|\nabla (a,u)(\tau)\|_{H^{s-1}(\mathbb{R}^3)}^2+\|\nabla v(\tau)\|_{H^{s}(\mathbb{R}^3)}^2\big)d\tau\leq C\varepsilon_0^2,
	\end{equation} 
	for any $t\in\mathbb{R}_+$.
	Additionally, when $(a_0-a_*,u_0,v_0)\in L^1(\mathbb{R}^3)$, then it holds that
	\begin{equation}\label{Upper}
		\|\nabla^j(a-a_*,u,v)(t)\|_{L^2(\mathbb{R}^3)}\leq C(1+t)^{-\frac{3}{4}-\frac{j}{2}}, \quad 0\leq j\leq s, 
	\end{equation}
	where the constant $C>0$ only depends on the initial data.
\end{thm}
	\begin{rem}\label{rem1}
	By the Sobolev inequality and \eqref{Upper}, it follows for any $p\in [2,6]$ and $0\leq j\leq s-1$ that
	\begin{equation}\label{Upper-p}
		\begin{aligned}
			\|\nabla^j(a-a_*,u,v)(t)\|_{L^p(\mathbb{R}^3)}&\leq  C(1+t)^{-\frac{3}{2}(1-\frac{1}{p})-\frac{j}{2}}.
		\end{aligned}
	\end{equation}
	Moreover, it holds for $0\leq j\leq s-2$ that
	\begin{equation}\label{Upper-inf}
		\begin{aligned}
			\|\nabla^j(a-a_*,u,v)(t)\|_{L^\infty(\mathbb{R}^3)}&\leq  C(1+t)^{-\frac{3+j}{2}}.
		\end{aligned}
	\end{equation}
\end{rem}

\begin{rem}
	Due to the damping structure arising from the drag force, the difference of velocities $(u-v)$ has a faster time decay rates satisfying
	\begin{align}
	\|(u-v)(t)\|_{L^2}\leq C(1+t)^{-\frac{5}{4}}.	
	\end{align}
\end{rem}
	It should be noted that the above time decay rates \eqref{Upper} are optimal. Indeed, we can obtain the lower bound of the time decay rates as follows.
	\begin{thm}\label{thm3}
		Assume the conditions in Theorem \ref{thm1} hold. If the Fourier transform $(\hat{\phi}_0(\xi),\hat{u}_0(\xi),\hat{v}_0(\xi))$ of the initial perturbation $(\phi_0,u_0,v_0) \triangleq (a_0-a_*,u_0,v_0)$ satisfies
		\begin{equation}\label{in-data-optimal}
			\inf_{ |\xi|<r_0}|\hat{\phi}_0(\xi)|\geq c_0>0, \quad \Big({\rm{I}}-\frac{\xi\xi^t}{|\xi|^2}\Big)\hat{u}_0=0, \quad  \inf_{ |\xi|<r_0 }|\hat{v}_0(\xi)|\geq c_0>0,
		\end{equation}
		where $c_0$ denotes a positive constant and $r_0$ is sufficiently small, then the global solution $(a,u,v)$ given by Theorem \ref{thm1} satisfies for large-time that
		\begin{equation}\label{optimal-est}
			d_*(1+t)^{-\frac{3}{4}-\frac{j}{2}}\leq \|\nabla^j(a-a_*,u,v)(t)\|_{L^2(\mathbb{R}^3)}\leq C(1+t)^{-\frac{3}{4}-\frac{j}{2}}, \quad 0\leq j\leq s, 
		\end{equation}
		where $d_*$ and $C$ are positive constants independent of time.
	\end{thm}
	
		Now we sketch the main ideas. The main difficulty to prove Theorem \ref{thm1} comes from the pressure term $\nabla P$ in the incompressible NS equations. To overcome the difficulty, we employ Hodge decomposition to separate $\nabla P$ into its linear and nonlinear components, i.e.,
		\begin{align}
			\nabla P=-c\nabla(-\Delta)^{-1}\text{div}(u-v)+\nabla(-\Delta)^{-1}\text{div}\Big(v\cdot\nabla v-c(e^\phi-1)(u-v)\Big),
		\end{align}
	where $c=\rho_*$. Then the system \eqref{Main2} around $(a_{*},0,0)$ is reduced to a perturbation system 
	\begin{equation}\label{Main5}
	\left\{
	\begin{aligned}
		&\partial_t\phi +{\rm{div}}u=f_1, \\
		&\partial_tu+\nabla \phi+u-v=f_2,\\
		&\partial_tv+cv-c\mathcal{J}u-\Delta v=f_3,
	\end{aligned}
	\right.
\end{equation}
	and the nonlinear terms $f_1,f_2,f_3$ satisfy
\begin{equation}
	\begin{aligned}
		f_1=-u\cdot\nabla\phi,~f_2=-u\cdot\nabla u,~~f_3=-\mathcal{J}(v\cdot\nabla v)+\mathcal{J}(c(e^{\phi}-1)(u-v)),
	\end{aligned}
\end{equation}
	with the initial data
	\begin{equation}
		(\phi,u,v)|_{t=0}=(\phi_0,u_0,v_0),
	\end{equation}
	where $\phi=a-a_*,~\phi_0=a_0-a_*$, and 
	\begin{align}\label{121101}
	\mathcal{J}\triangleq{\rm{I}}+\nabla(-\Delta)^{-1}\text{div}.
	\end{align}
By Duhamel's principle, we obtain the solution $U=(\phi,u,v)^t$ of \eqref{Main5} as follows,
\begin{align}\label{121501}
U(t,x)=G*U_0+\int_0^tG(t-\tau)*F(\tau)d\tau,
\end{align}
where $G(t,x)$ is the Green function for the linear part of \eqref{Main5} and $F=(f_1,f_2,f_3)^t$. We first carefully analyze the Green function $G(t,x)$ and obtain its time decay rates, then we use the formula \eqref{121501} and energy method to prove the global existence of solution, and further obtain
	\begin{equation}\label{HT11}
		\|(\phi,u,v)\|_{H^s}\leq C(1+t)^{-\frac{3}{4}}, \quad s\geq 3.
	\end{equation}
It should be mentioned that the above decay rates for high-order derivatives are slow. To improve the decay rates, we decompose the solution into low-frequency  and high-frequency part, and then obtain 
	\begin{equation}\label{JU2}
		\frac{d}{dt}\|\nabla (\phi,u,v)\|_{H^{s-1}}^2+C_2\|\nabla (\phi,u,v)\|_{H^{s-1}}^2\leq C\|\nabla(\phi^\ell,u^\ell,v^\ell)\|_{L^2}^2,
	\end{equation}
where $(\phi^\ell,u^\ell,v^\ell)$ denotes the low-frequency part. That is the time decay rates of $\|\nabla(\phi,u,v)\|_{H^{s-1}}$  is dominated by low-frequency part. We use spectral analysis to get
	\begin{equation}\label{JU1}
		\|\nabla(\phi^\ell,u^\ell,v^\ell)\|_{L^2}\leq C(1+t)^{-\frac{5}{4}}.
	\end{equation}
Substituting \eqref{JU1} into \eqref{JU2} and using Gr$\ddot{\text{o}}$nwall's inequality yield 
	\begin{equation}\label{HT1}
		\|\nabla(\phi,u,v)\|_{H^{s-1}}\leq C(1+t)^{-\frac{5}{4}},
	\end{equation}
which is indeed better than those in \eqref{HT11}. Similarly, we can get better decay estimates for higher order derivatives of the solution. Finally, we have
	\begin{equation}\label{HT2}
	\|\nabla^j(\phi,u,v)\|_{L^2}\leq C(1+t)^{-\frac{3}{4}-\frac{j}{2}},\quad 0\leq j\leq s.
\end{equation}

The remaining task is to prove the rates in \eqref{HT2} are optimal by establishing lower bound decay estimates. Firstly, we show that for some constant $c_0>0$,
\begin{equation}\label{HT3}
\|(\bar{\phi},\bar{u},\bar{v})\|_{L^2}\geq c_0(1+t)^{-\frac{3}{4}},
\end{equation}
where $(\bar{\phi},\bar{u},\bar{v})=G*\bar{U}_0$ is a solution of linear equations of \eqref{Main5} with a special initial data $\bar{U}_0$. Secondly, we use \eqref{121501} and the upper bound \eqref{HT2} to achieve
\begin{align}
\Big\|\int_0^tG(t-\tau)*F(\tau)d\tau\Big\|_{L^2}\leq C(1+t)^{-\frac{3}{4}}(\varepsilon_0^2+\varepsilon_0 \mathcal{I}_0).
\end{align}
Due to the smallness of $\varepsilon_0$ and \eqref{HT3}, the triangle inequality implies 
\begin{equation}\label{Ht4}
\|(\phi,u,v)\|_{L^2}\geq \frac{1}{2}c_0(1+t)^{-\frac{3}{4}}.
\end{equation}
On other hand, we can prove $\|\Lambda^{-1}(\phi,u,v)\|_{L^2}\leq C(1+t)^{-\frac{1}{4}}$, which together with \eqref{Ht4} and for any $0\leq j\leq s$,
$$\|(\phi,u,v)\|_{L^2}\leq C\|\Lambda^{-1}(\phi,u,v)\|_{L^2}^{\frac{j}{j+1}}\|\nabla^j(\phi,u,v)\|_{L^2}^{\frac{1}{j+1}},
$$
yields that
	\begin{equation}
	\|\nabla^j(\phi,u,v)\|_{L^2}\geq C_*(1+t)^{-\frac{3}{4}-\frac{j}{2}}.
\end{equation}	
	
There also has been important progress on the well-posedness and dynamic behaviors of the solutions to the Euler-Navier-Stokes system and related models. We refer to \cite{MR3546341,MR4175837,Wuzhou2021,MR1643525,MR4316126,Choi-Jung2021} and the references therein.

The rest of the paper is organized as follows. In Section \ref{Sec2}, we introduce some notations and auxiliary lemmas used in the proof of the main results. Section \ref{Sec3} is related to the \emph{a priori} estimates which can extend the local solution to a global one. In Section \ref{Sec4}, we investigate the large-time behavior of the solutions. Finally, the proof of Theorems \ref{thm1}-\ref{thm3} will be given in Section \ref{Sec5}.
	
	\section{Preliminaries}\label{Sec2}
	\subsection{Notation}
	\hspace{2em}In this section, we first introduce the notation and conventions used throughout the paper.
	$L^p(\mathbb{R}^3)$ and $W^{k,p}(\mathbb{R}^3)$ denote the usual Lebesgue and Sobolev space on $\mathbb{R}^3$, with norms $\|\cdot\|_{L^p}$ and $\|\cdot\|_{W^{k,p}}$, respectively.  When $p=2$, we denote $W^{k,p}(\mathbb{R}^3)$ by $H^k(\mathbb{R}^3)$ with the norm $\|\cdot\|_{H^k}$, and set
	$$
	\|u\|_{H^k(\mathbb{R}^3)}=\|u\|_{H^k},\quad \|u\|_{L^p(\mathbb{R}^3)}=\|u\|_{L^p}.
	$$
	We denote by $C$ a generic positive constant which may vary in different estimates. $f_1\lesssim f_2$ describes that there exists a constant $C>0$ such that $f_1\leq C f_2$. The symbol $f_1\sim f_2$ represents the functions $f_1$ and $f_2$ are equivalent, which means that there exist positive constants $C_1, C_2$ such that $f_1\leq C_1f_2$ and $f_2\leq C_2f_1$. For an integer $k$,  the symbol $\nabla^k$ denotes the summation of all terms $D^\ell=\partial_{x_1}^{\ell_1}\partial_{x_2}^{\ell_2}\partial_{x_3}^{\ell_3}$ with the multi-index $\ell$ satisfying $|\ell|=\ell_1+\ell_2+\ell_3=k$. For a function $f$, $\|f\|_{X}$ denotes the norm of $f$ on $X$. $\|(f,g)\|_{X}$ denotes $\|f\|_{X}+\|g\|_{X}$. The Fourier transform of $f$ is denoted by $\hat{f}$ or $\mathscr{F}[f]$ satisfying
	\begin{equation}\hat{f}(\xi)
		=\mathscr{F}[f](\xi)
		=(2\pi)^{-\frac{3}{2}}\int_{\mathbb{R}^3} f(x)e^{-ix\cdot\xi}dx,\quad \xi\in\mathbb{R}^3.
	\end{equation}
	Let $\Lambda^k$ be the pseudodifferential operator defined by
	\begin{equation}\label{Lambda}
		\Lambda^k f=\mathscr{F}^{-1}\Big(|\xi|^k \hat{f}(\xi)\Big)\ \ \text{for}\ k\in \mathbb{R}.
	\end{equation}
	We define operators $\mathcal{K}_{1}$ and $\mathcal{K}_{\infty}$ on $L^{2}$ by
	\begin{equation}\label{THZ5}
		\mathcal{K}_{1}f=f^\ell=\mathscr{F}^{- 1}\big(\hat\chi_{1}(\xi)\mathscr{F}[f](\xi)\big),\quad 	\mathcal{K}_{\infty}f=f^h=\mathscr{F}^{- 1}\big(\hat\chi_{\infty}(\xi)\mathscr{F}[f](\xi)\big),
	\end{equation}
	where $\hat{\chi}_{j}(\xi)(j=1,\infty)\in C^{\infty}(\mathds{R}^{3})$, $0\leq \hat\chi_{j}\leq 1$ are smooth cut-off functions defined by
	\begin{equation*}
		\hat\chi_{1}(\xi)=\left\{
		\begin{array}{l}
			1\quad (|\xi|\leq r_{0}),\\
			0\quad (|\xi|\geq R_0),
		\end{array}
		\right.
		\quad\hat\chi_{\infty}(\xi)=1-\hat\chi_{1}(\xi),
	\end{equation*}
	where the positive constant $r_0$ is sufficiently small and $R_0$ is sufficiently large.
	
	To analyze the large-time behavior of the solutions $U(t,x)=(\phi,u,v)^t$ in frequency space, we adopt the low-high frequency decomposition for the solution: 
	\begin{equation}\label{l-h-com}
		U(t,x)=U^{\ell}(t,x)+U^{h}(t,x)\triangleq(\phi^{\ell},u^{\ell},v^{\ell})+(\phi^{h},u^{h},v^{h}),
	\end{equation}
	where $U^{\ell}(t,x)\triangleq\mathcal{K}_1U(t,x)$ is the low-frequency part and $U^{h}(t,x)\triangleq\mathcal{K}_{\infty}U(t,x)$ represents the high-frequency part. The operators $\mathcal{K}_1$ and $\mathcal{K}_{\infty}$ has been introduced in \eqref{THZ5}.
	\subsection{Auxiliary lemmas}
	\hspace{2em}In this subsection, we introduce some elementary inequalities and auxiliary lemmas that are used extensively in the proof of the main theorems in this paper. 
	\begin{lem}(\cite[Lemma A.3]{MR2917409})\label{lema2}
		Let $m\geq 1$ be an integer and define the communicator 
		\begin{equation}
			[\nabla^m,f]g=\nabla^{m}(fg)-f\nabla^m g.
		\end{equation}
		Then we have
		\begin{equation}
			\big\|[\nabla^m,f]g\big\|_{L^p}\lesssim \|\nabla f\|_{L^{p_1}}\|\nabla^{m-1}g\|_{L^{p_2}}+
			\|\nabla^m f\|_{L^{p_3}}\|g\|_{L^{p_4}},
		\end{equation}
		where $p, p_2, p_3\in(1,+\infty)$ and 
		\begin{equation}
			\frac{1}{p}=\frac{1}{p_1}+\frac{1}{p_2}=\frac{1}{p_3}+\frac{1}{p_4}.
		\end{equation}
	\end{lem}
	\begin{lem} \label{Ap2} (Gagliardo-Nirenberg inequality, \cite{MR109940} or \cite[Lemma A.1]{MR2917409})\label{lema3}
		Let $l,s$ and $k$ be any real numbers satisfying $0\leq
		l,s<k$, and let  $p, r, q \in [1,\infty]$ and $\frac{l}{k}\leq
		\theta\leq 1$ such that
		$$\frac{l}{3}-\frac{1}{p}=\left(\frac{s}{3}-\frac{1}{r}\right)(1-\theta)+\left(\frac{k}{3}-\frac{1}{q}\right)\theta.
		$$ Then, for any $u\in W^{k,q}(\mathbb{R}^3),$  we have
		\begin{equation}
			\label{h20}
			\|\nabla^l u\|_{L^p}\lesssim \|\nabla^s u\|_{L^r}^{1-\theta}\|\nabla^k
			u\|_{L^q}^{\theta}.
		\end{equation}	
	\end{lem}
	
	
	\begin{lem}\label{lema5}
		Let $a\geq 0$ and integer $l \geq 0$, then we have
		\begin{equation}
			\|\nabla^l f\|_{L^2}\lesssim \|\nabla^{l+1}f\|_{L^2}^{1-\theta}\|\Lambda^{-a}f\|_{L^2}^\theta,\quad \text{where}~\theta=\frac{1}{1+l+a}.
		\end{equation}
	\end{lem}
	\begin{proof}		
		According to the Parseval's equality, the definition of $\Lambda^{-a}f$ and H\"{o}lder's inequality, we get
		\begin{equation}
			\|\nabla^l f\|_{L^2}=\Big\||\xi|^l\hat{f}\Big\|_{L^2}
			\lesssim\big\||\xi|^{l+1}\hat{f}\big\|_{L^2}^{1-\theta}\big\||\xi|^{-a}\hat{f}\big\|_{L^2}^\theta=\|\nabla^{l+1}f\|_{L^2}^{1-\theta}\|\Lambda^{-a}f\|_{L^2}^\theta,
		\end{equation}
		for $\theta=\frac{1}{1+l+a}$. Hence, this completes the proof of this lemma.
	\end{proof}
	
	\begin{lem}\label{lema6}
		For $0\leq k<m$, there exist positive a constant $C$ such that for $f\in H^m$,
		\begin{equation}\label{h-l-est1}
			\|\nabla^{m} f^\ell\|_{L^2}\leq C\|\nabla^{k} f^\ell\|_{L^2},\quad \|\nabla^{k}f^h\|_{L^2}\leq C\|\nabla^{m}f^h\|_{L^2},
		\end{equation}
		and 
		\begin{equation}\label{h-l-est2}
			\|\nabla^{k} f^\ell\|_{L^2}\leq C\|\nabla^{k} f\|_{L^2},\quad\|\nabla^{k}f^h\|_{L^2}\leq C\|\nabla^{k}f\|_{L^2}.
		\end{equation}
	\end{lem}
	\begin{proof}
		According to Parseval's equality, there exists a positive constant $C$ such that for $k<m$
		$$
		\big\|\nabla^m f^\ell\big\|_{L^2}=\big\|(i\xi)^m\hat{\chi}_1(\xi)\hat{f}(\xi)\big\|_{L^2}\leq C \big\| |\xi|^{m-k}(i\xi)^k\hat{\chi}_1(\xi)\hat{f}(\xi)\big\|_{L^2}\leq C \big\|(i\xi)^k\hat{\chi}_1(\xi)\hat{f}(\xi)\big\|_{L^2}=C\|\nabla^kf^\ell\|_{L^2},
		$$
		and
		$$
		\big\|\nabla^kf^h\big\|_{L^2}= \big\|(i\xi)^k\hat{\chi}_\infty(\xi)\hat{f}(\xi)\big\|_{L^2}\leq C \big\| |\xi|^{k-m}(i\xi)^{m}\hat{\chi}_\infty(\xi)\hat{f}(\xi)\big\|_{L^2}\leq C \big\|(i\xi)^{m}\hat{\chi}_\infty(\xi)\hat{f}(\xi)\big\|_{L^2}=C\|\nabla^mf^h\|_{L^2}.
		$$
		Since the cut-off functions $\hat{\chi}_1(\xi)$ and $\hat{\chi}_\infty(\xi)$ are bounded, we also have
		$$
		\big\|\nabla^k f^\ell\big\|_{L^2}=\big\|(i\xi)^k\hat{\chi}_1(\xi)\hat{f}(\xi)\big\|_{L^2}\leq C\big\| (i\xi)^k \hat{f}(\xi)\big\|_{L^2}=C\|\nabla^kf\|_{L^2},
		$$
		and
		$$
		\big\|\nabla^k f^h\big\|_{L^2}=\big\|(i\xi)^k\hat{\chi}_\infty(\xi)\hat{f}(\xi)\big\|_{L^2}\leq C\big\| (i\xi)^k \hat{f}(\xi)\big\|_{L^2}=C\|\nabla^kf\|_{L^2}.
		$$
		Therefore we complete the proof of this lemma.  
	\end{proof}
	
	\section{The  \emph{a priori} estimates}\label{Sec3}
	\hspace{2em} In this section, we aim to establish the  \emph{a priori} estimates of classical solution for the nonlinear system \eqref{Main2}-\eqref{in-data}. First, we reformulate the original system into perturbed form and prove the local existence.
	\subsection{Reformulated system and local existence} 
	\label{sub:reformulated_problem}
\hspace{2em}	For notation convenience, we denote the perturbation of the density below
	\begin{equation}\label{new-var}
		\phi=a-a_*,\quad c=e^{a_*}=\rho_*.
	\end{equation}
	Then, the system  \eqref{Main2}-\eqref{in-data} can be reformulated to
	\begin{equation}\label{3-1}
		\left\{
		\begin{aligned}
			&\partial_t\phi+u\cdot\nabla \phi+\text{div}u=0,\\
			&\partial_tu+u\cdot\nabla u+u-v+\nabla \phi=0,\\
			&\partial_tv+v\cdot\nabla v+\nabla P=\Delta v+c(e^{\phi}-1)(u-v)+c(u-v),\\
			&\text{div}v=0,
		\end{aligned}
		\right.
	\end{equation}
	with the initial data 
	\begin{equation}\label{3-2}
		(\phi,u,v)|_{t=0}=(\phi_0,u_0,v_0),
	\end{equation}
	and far fields states 
	\begin{equation}\label{3-21}
		(\phi,u,v)\rightarrow (0,0,0),\quad \text{as}\quad |x|\rightarrow+\infty.
	\end{equation}
	
We establish the following local existence theorem of the classical solution of the system \eqref{3-1}-\eqref{3-21}, which can be proved similarly as that in \cite{MR564670} by using contraction mapping principle. Here we directly give the main result and omit the details of the proof.
	\begin{thm}(Local existence)\label{lem-loc}
		Assume $(\phi_0,u_0,v_0)\in H^s(\mathbb{R}^3)$ for an integer $s \geq 3$ and ${\rm{div}}v_0=0$, then there exists a short time $T_0>0$ such that the reformulated system \eqref{3-1}-\eqref{3-21} admits a unique classical solution $(\phi,u,v)$ satisfying
		\begin{equation}
			\begin{aligned}
				&\phi\in C([0,T_0],H^s(\mathbb{R}^3))\cap C^1([0,T_0],H^{s-1}(\mathbb{R}^3)),\\
				&u\in C([0,T_0],H^s(\mathbb{R}^3))\cap C^1([0,T_0],H^{s-1}(\mathbb{R}^3)),~\nabla u\in L^2([0,T_0],H^{s-1}(\mathbb{R}^3)),\\
				&v\in C([0,T_0],H^s(\mathbb{R}^3))\cap C^1([0,T_0],H^{s-2}(\mathbb{R}^3)),~\nabla v\in L^2([0,T_0],H^{s}(\mathbb{R}^3)).
			\end{aligned}
		\end{equation}
	\end{thm}
	
	Then, we intend to establish uniform estimates for extending the local-in-time classical solution to a global one.  Therefore, we provide the \emph{a priori} assumption for any given time $T>0$ 
	\begin{equation}\label{a priori est}
		\sup_{0< t\leq T}\|(\phi,u,v)(t)\|_{H^s}\leq \delta,
	\end{equation}
	where $s\geq 3$ is an integer and $\delta>0$ is a sufficiently small constant.
	
	
	\subsection{Time-independent energy estimates} 
	\label{sub:a_priori_estimates}
	\hspace{2em}In this subsection, we plan to establish the energy estimates for the  reformulated system \eqref{3-1}-\eqref{3-21}.
	\begin{lem}\label{lem-n}
		Let $T$ be any given positive constant. Assume the conditions in Theorem \ref{thm1} hold. Let $(\phi,u,v)$ be the classical solution of the system \eqref{3-1}-\eqref{3-21} satisfying the a priori assumption \eqref{a priori est}, then it holds for $0<t\leq T$ and $s\geq 3$,
		\begin{equation}
			\begin{aligned}\label{lem-n1}
				&\frac{1}{2}\frac{d}{dt}(\|\phi\|_{H^s}^2+\| u\|_{H^s}^2+\frac{1}{c}\|v\|_{H^s}^2)
				+\|u-v\|_{H^s}^2+\frac{1}{c}\|\nabla v\|_{H^s}^2\\
				&\leq C\delta (\|\nabla \phi\|_{H^{s-1}}^2+\| \nabla u\|_{H^{s-1}}^2+
				\|u-v\|_{H^{s}}^2+\|\nabla v\|_{H^{s-1}}^2),
			\end{aligned}
		\end{equation}
	where $C$ is a positive constant independent of time.	
	\end{lem}
	\begin{proof}
		Multiplying $\eqref{3-1}_1$ by $\phi$, $\eqref{3-1}_2$ by $u$, and  $\eqref{3-1}_3$ by $\frac{1}{c}v$, respectively, integrating the resulting equations in $\mathbb{R}^3$, summing them up, we obtain 
		\begin{equation}\label{THZ6}
			\begin{aligned}
				&\frac{1}{2}\frac{d}{dt}(\|\phi\|_{L^2}^2+\|u\|_{L^2}^2+\frac{1}{c}\|v\|_{L^2}^2)+\|u-v\|_{L^2}^2+\frac{1}{c}\|\nabla v\|_{L^2}^2\\
				&=-\int_{{\mathbb{R}^3}}(u\cdot\nabla  \phi)\phi dx-\int_{\mathbb{R}^3}( u\cdot\nabla u)\cdot udx+\int_{\mathbb{R}^3}(e^\phi-1)(u-v)\cdot vdx.	
			\end{aligned}
		\end{equation}
		By H\"{o}lder's inequality, Sobolev's inequalities, and  \eqref{a priori est}, the right-hand side of \eqref{THZ6} can be bounded by
		$$
		\begin{aligned}
			&\Big|\int_{{\mathbb{R}^3}}(u\cdot\nabla  \phi)\phi dx\Big|+\Big|\int_{\mathbb{R}^3} (u\cdot \nabla u)\cdot udx\Big|+\Big|\int_{\mathbb{R}^3}(e^\phi-1)(u-v)\cdot vdx\Big|\\	
			&\leq  \|u\|_{L^3}\|\nabla\phi\|_{L^2}\| \phi\|_{L^6}+\|u\|_{L^6}\|\nabla u\|_{L^2}\| u\|_{L^3}+\|e^\phi-1\|_{L^3}\|u-v\|_{L^2}\|v\|_{L^6}\\
			&\leq C \|u\|_{H^1}\|\nabla \phi\|_{L^2}^2+C\|\nabla u\|_{L^2}^2\| u\|_{H^1}+C\|\phi\|_{H^1}\|u-v\|_{L^2}\|\nabla v\|_{L^2}\\
			&\leq C\delta(\|\nabla \phi\|_{L^2}^2+\|\nabla u\|_{L^2}^2+\|u-v\|_{L^2}^2+\|\nabla v\|_{L^2}^2).
		\end{aligned}
		$$
		Thus we have
		\begin{equation}\label{L21}
			\begin{aligned}
				&\frac{1}{2}\frac{d}{dt}(\|\phi\|_{L^2}^2+\|u\|_{L^2}^2+\frac{1}{c}\|v\|_{L^2}^2)+\|u-v\|_{L^2}^2+\frac{1}{c}\|\nabla v\|_{L^2}^2\\
				&\leq C\delta(\|\nabla \phi\|_{L^2}^2+\|\nabla u\|_{L^2}^2+\|u-v\|_{L^2}^2+\|\nabla v\|_{L^2}^2).
			\end{aligned}
		\end{equation}	
		For $1\leq k\leq s$, applying the operator $\nabla^k$ to $\eqref{3-1}_1$ , $\eqref{3-1}_2$ and $\eqref{3-1}_3$  , multiplying the resulting equations by $\nabla^k \phi$ , $\nabla^k u$ and $\frac{1}{c}\nabla^k v$ respectively, and integrating them in $\mathbb{R}^3$, we obtain
		\begin{equation}\label{lem-w2}
			\begin{aligned}
				&\frac{1}{2}\frac{d}{dt}(\|\nabla^k \phi \|_{L^2}^2+\|\nabla^k u\|_{L^2}^2+\frac{1}{c}\|\nabla^kv\|_{L^2}^2)
				+\|\nabla^k(u-v)\|_{L^2}^2+\frac{1}{c}\|\nabla\nabla^kv\|_{L^2}^2\\
				&=-\int_{\mathbb{R}^3}\nabla^k (u\cdot\nabla \phi)\cdot\nabla^k \phi dx-\int_{\mathbb{R}^3} \nabla^k(u\cdot\nabla u)\cdot \nabla^k udx\\
				&~~-\frac{1}{c}\int_{\mathbb{R}^3}\nabla^k(v\cdot\nabla v)\cdot\nabla^kvdx+\int_{\mathbb{R}^3}\nabla^k((e^{\phi}-1)(u-v))\cdot\nabla^kvdx.
			\end{aligned}
		\end{equation}
		By Lemma \ref{lema2}, the first term on the right-hand side of \eqref{lem-w2} is calculated as follows
		\begin{equation}\label{lem-n2-1}
			\begin{aligned}
				&\Big| \int_{\mathbb{R}^3}\nabla^k (u\cdot\nabla \phi)\cdot\nabla^k \phi dx\Big|\\
				&\leq \Big| \int_{\mathbb{R}^3} (\nabla^k(u\cdot\nabla \phi)-u\cdot\nabla^k\nabla \phi)\cdot\nabla^k \phi dx\Big|+\Big|\int_{\mathbb{R}^3} u\cdot\nabla\nabla^k \phi\cdot\nabla^k \phi dx\Big|\\
				&\leq\|[\nabla^k,u]\nabla\phi\|_{L^2}\|\nabla^k \phi\|_{L^2} +C\|\text{div} u\|_{L^\infty}\|\nabla^k \phi \|_{L^2}^2\\
				&\leq C(\|\nabla u\|_{L^\infty}\|\nabla^k \phi\|_{L^2}+\|\nabla \phi\|_{L^\infty}\|\nabla^k u\|_{L^2})\|\nabla^k \phi\|_{L^2}+C\|\text{div} u\|_{L^\infty}\|\nabla^k \phi \|_{L^2}^2\\
				&\leq C(\|\nabla^2u\|_{H^1}\|\nabla^k \phi\|_{L^2}+\|\nabla^2 \phi\|_{H^1}\|\nabla^k u\|_{L^2})\|\nabla^k \phi\|_{L^2}+C\|\nabla^2 u\|_{H^1}\|\nabla^k \phi\|_{L^2}^2\\
				&\leq C\delta(\|\nabla^k \phi\|_{L^2}^2+\|\nabla^k u\|_{L^2}^2).
			\end{aligned}
		\end{equation}
		Similarly, the second term on the right-hand side of \eqref{lem-w2} is estimated  as follows
		\begin{equation}\label{lem-w3}
			\begin{aligned}
				&\Big| \int_{\mathbb{R}^3}\nabla^k (u\cdot\nabla u)\cdot\nabla^k u dx\Big|\\
				&\leq \Big| \int_{\mathbb{R}^3} (\nabla^k(u\cdot\nabla u)-u\cdot\nabla^k\nabla u)\cdot\nabla^k udx\Big| +\Big| \int_{\mathbb{R}^3} u\cdot\nabla\nabla^k u\cdot\nabla^k udx\Big| \\
				&\leq \|[\nabla^k,u]\nabla u\|_{L^2}\|\nabla^k u\|_{L^2}
				+C\|\text{div} u\|_{L^\infty}\|\nabla^k u\|_{L^2}^2\\
				&\leq C\|\nabla u\|_{L^\infty}\|\nabla^k u\|_{L^2}^2
				+C\|\text{div} u\|_{L^\infty}\|\nabla^k u\|_{L^2}^2\\
				&\leq C\delta\|\nabla^k u\|_{L^2}^2.
			\end{aligned}
		\end{equation}
		We now turn to the estimate for the third term on the right-hand side of \eqref{lem-w2},
		\begin{equation}
			\begin{aligned}
				&\Big| \int_{\mathbb{R}^3}\nabla^k (v\cdot\nabla v)\cdot\nabla^k v dx\Big|\\
				&\leq \Big| \int_{\mathbb{R}^3} (\nabla^k(v\cdot\nabla v)-v\cdot\nabla^k\nabla v)\cdot\nabla^k vdx\Big| +\Big| \int_{\mathbb{R}^3} v\cdot\nabla\nabla^k v\cdot\nabla^k vdx\Big| \\
				&\leq \|[\nabla^k,v]\nabla v\|_{L^2}\|\nabla^k v\|_{L^2}
				+C\|\text{div} v\|_{L^\infty}\|\nabla^k v\|_{L^2}^2\\
				&\leq C\|\nabla v\|_{L^\infty}\|\nabla^k v\|_{L^2}^2\\
				&\leq C\delta\|\nabla^k v\|_{L^2}^2.
			\end{aligned}
		\end{equation}
		In terms of Lemma \ref{lema2}, the last term on the right-hand side of \eqref{lem-w2} is stated below
		\begin{equation}
			\begin{aligned}
				&\Big| \int_{\mathbb{R}^3}\nabla^k ( (e^{\phi}-1)(u-v))\cdot\nabla^k v dx\Big|\\
				&\leq \Big| \int_{\mathbb{R}^3} (\nabla^k(e^{\phi}-1)(u-v)-(e^{\phi}-1)\nabla^k(u-v))\cdot\nabla^k vdx\Big|\\
				&\quad+\Big| \int_{\mathbb{R}^3} (e^{\phi}-1)\nabla^k(u-v)\cdot\nabla^k vdx\Big| \\
				&\leq \|[\nabla^k, e^{\phi}-1](u-v)\|_{L^2}\|\nabla^kv\|_{L^2}
				+\|e^{\phi}-1\|_{L^\infty}\|\nabla^k (u-v)\|_{L^2}\|\nabla^kv\|_{L^2}\\
				&\leq C(\|\nabla (e^{\phi}-1)\|_{L^3}\|\nabla^{k-1}(u-v)\|_{L^6}
				+\|\nabla^k(e^{\phi}-1)\|_{L^2}\|u-v\|_{L^\infty})\|\nabla^kv\|_{L^2}\\
				&\quad+C\|\nabla\phi\|_{H^1}\|\nabla^k (u-v)\|_{L^2}\|\nabla^kv\|_{L^2}\\
				&\leq C\delta(\|\nabla^k \phi \|_{L^2}^2+\|\nabla^{k} (u-v)\|_{L^2}^2+
				\|\nabla^kv\|_{L^2}^2).
			\end{aligned}
		\end{equation}
		Therefore, we conclude from the above estimates to prove
		\begin{equation}\label{A1}
			\begin{aligned}
				&\frac{1}{2}\frac{d}{dt}(\|\nabla^k \phi \|_{L^2}^2+\|\nabla^k u\|_{L^2}^2+\frac{1}{c}\|\nabla^kv\|_{L^2}^2)
				+\|\nabla^k(u-v)\|_{L^2}^2+\frac{1}{c}\|\nabla\nabla^kv\|_{L^2}^2\\
				&\leq C\delta (\|\nabla^k \phi \|_{L^2}^2+\|\nabla^ku\|_{L^2}^2+\|\nabla^{k} (u-v)\|_{L^2}^2+\|\nabla^kv\|_{L^2}^2).
			\end{aligned}
		\end{equation}
		Summing $k$ from 1 to s and combining \eqref{L21} yield
		\begin{equation}
			\begin{aligned}
				&\frac{1}{2}\frac{d}{dt}(\|\phi\|_{H^s}^2+\| u\|_{H^s}^2+\frac{1}{c}\|v\|_{H^s}^2)
				+\|u-v\|_{H^s}^2+\frac{1}{c}\|\nabla v\|_{H^s}^2\\
				&\leq C\delta (\|\nabla \phi\|_{H^{s-1}}^2+\| \nabla u\|_{H^{s-1}}^2+
				\|u-v\|_{H^{s}}^2+\|\nabla v\|_{H^{s-1}}^2).
			\end{aligned}
		\end{equation}
		This completes the proof of the lemma.
	\end{proof}
	
	\begin{lem}\label{lem-nx}
			Let $T$ be any given positive constant. Assume the conditions in Theorem \ref{thm1} hold. Let $(\phi,u,v)$ be the classical solution of the system \eqref{3-1}-\eqref{3-21} satisfying the a priori assumption \eqref{a priori est}, then it holds for $0<t\leq T$ and $s\geq 3$,
		\begin{equation}\label{lem-nx1}
			\begin{aligned}
				&\frac{1}{2}\|\nabla \phi\|_{H^{s-1}}^2+\sum_{k=1}^s\frac{d}{dt}\int_{\mathbb{R}^3} \nabla^{k-1}u\cdot\nabla^k \phi dx\\
				&\leq C(\|\nabla u\|_{H^{s-1}}^2+\|u-v\|_{H^{s-1}}^2),
			\end{aligned}
		\end{equation}	
		where $C$ is a positive constant independent of time.
	\end{lem}
	\begin{proof}
		For $1\leq k\leq s$, applying the operator $\nabla^{k-1}$ to $\eqref{3-1}_2$, multiplying the resulting equation by $\nabla^{k} \phi$, and integrating them in $\mathbb{R}^3$, we obtain
		\begin{equation}\label{lem-nx2}
			\begin{aligned}
				&\|\nabla^k \phi\|_{L^2}^2+\frac{d}{dt}\int_{\mathbb{R}^3} \nabla^{k-1}u\cdot\nabla^k \phi dx\\
				&=-\int_{\mathbb{R}^3}\nabla^{k-1}u\cdot\nabla^k\text{div}udx-\int_{\mathbb{R}^3}\nabla^{k-1}(u-v)\cdot\nabla^k \phi  dx\\
				&\quad-\int_{\mathbb{R}^3}\nabla^k(u\cdot\nabla \phi)\cdot \nabla^{k-1}udx-\int_{\mathbb{R}^3}\nabla^{k-1}(u\cdot\nabla u)\cdot \nabla^k \phi dx.
			\end{aligned}
		\end{equation}
		The first and second terms on the right-hand side of \eqref{lem-nx2} are estimated as 
		$$
		\begin{aligned}
			&\Big|\int_{\mathbb{R}^3}\nabla^{k-1}u\cdot\nabla^k\text{div}udx\Big|+\Big|\int_{\mathbb{R}^3}\nabla^{k-1}(u-v)\cdot\nabla^k \phi ~dx\Big|\\
			&\leq \|\nabla^ku\|_{L^2}^2+\|\nabla^{k-1}(u-v)\|_{L^2}^2+\frac{1}{4}\|\nabla^k \phi \|_{L^2}^2.
		\end{aligned}
		$$
		By H\"{o}lder's inequality and Lemma \ref{lema2}, we apply integration by parts to estimate the third term,
		$$
		\begin{aligned}
			&	\Big|\int_{\mathbb{R}^3}\nabla^k(u\cdot\nabla \phi)\cdot \nabla^{k-1}udx\Big|\\
			&\leq \Big|\int_{\mathbb{R}^3}(\nabla^{k}(u\cdot\nabla \phi)-u\cdot\nabla^k\nabla \phi)\cdot \nabla^{k-1}udx\Big|+\Big|\int_{\mathbb{R}^3}
			(u\cdot \nabla^{k}\nabla\phi)\cdot \nabla^{k-1}u dx\Big|\\
			&\leq \|[\nabla^k,u]\nabla\phi\|_{L^2}\|\nabla^{k-1}u\|_{L^2}+\|u\|_{L^\infty}\|\nabla^k \phi \|_{L^2}\|\nabla^ku\|_{L^2}+\|\nabla u\|_{L^3}\|\nabla^k\phi\|_{L^2}\|\nabla^{k-1}u\|_{L^6}\\
			&\leq C(\|\nabla u\|_{L^\infty}\|\nabla^{k}\phi\|_{L^2}+\|\nabla^{k}u\|_{L^2}\|\nabla \phi\|_{L^\infty}) \|\nabla^{k-1}u\|_{L^2}+C\|\nabla u\|_{H^1}\|\nabla^k \phi \|_{L^2}\|\nabla^ku\|_{L^2}\\
			&\leq C\delta(\|\nabla^2 u\|_{H^1}\|\nabla^{k}\phi\|_{L^2}+\|\nabla^{k}u\|_{L^2}\|\nabla^2 \phi\|_{H^1})+C\delta\|\nabla^k \phi\|_{L^2}\|\nabla^ku\|_{L^2}\\
			&\leq C\delta(\|\nabla^k \phi \|_{L^2}^2+\|\nabla^{k}u\|_{L^2}^2
			+\|\nabla^2\phi\|_{H^1}^2+\|\nabla^2 u\|_{H^1}^2).
		\end{aligned}
		$$
		In a similar way, we also have 
		$$
		\Big|\int_{\mathbb{R}^3}\nabla^k \phi \cdot\nabla^{k-1}(u\cdot\nabla u)dx\Big|
		\leq  C\delta(\|\nabla^k \phi \|_{L^2}^2+\|\nabla^{k}u\|_{L^2}^2
		+\|\nabla^2\phi\|_{H^1}^2+\|\nabla^2 u\|_{H^1}^2).
		$$
		It then follows from the above estimates to show
		\begin{equation}
			\begin{aligned}
				&\|\nabla^k \phi\|_{L^2}^2+\frac{d}{dt}\int_{\mathbb{R}^3} \nabla^{k-1}u\cdot\nabla^k \phi dx\\
				&\leq \|\nabla^ku\|_{L^2}^2+\|\nabla^{k-1}(u-v)\|_{L^2}^2+\frac{1}{4}\|\nabla^k \phi \|_{L^2}^2\\
				&\quad+C\delta(\|\nabla^k \phi \|_{L^2}^2+\|\nabla^{k}u\|_{L^2}^2
				+\|\nabla^2\phi\|_{H^1}^2+\|\nabla^2 u\|_{H^1}^2).
			\end{aligned}
		\end{equation}
		For $s\geq 3$, summing $k$ from 1 to s and using the smallness of the $\delta$, we obtain 
		\begin{equation}
			\begin{aligned}
				&\frac{1}{2}\|\nabla \phi\|_{H^{s-1}}^2+\sum_{k=1}^s\frac{d}{dt}\int_{\mathbb{R}^3} \nabla^{k-1}u\cdot\nabla^k \phi dx\\
				&\leq C(\|\nabla u\|_{H^{s-1}}^2+\|u-v\|_{H^{s-1}}^2).
			\end{aligned}
		\end{equation}
		Therefore, we complete the proof of Lemma \ref{lem-nx}.
	\end{proof}
	
	We are now in a position to establish uniform estimates as follows.
	\begin{prop}\label{pro-est}
		Assume the conditions in Theorem \ref{thm1} hold. Let $(\phi,u,v)$ be the classical solution of the system \eqref{3-1}-\eqref{3-21} satisfying the a priori assumption \eqref{a priori est}, then it holds
		\begin{equation}\label{pro-est1}
			\|(\phi,u,v)(t)\|_{H^s}^2+\int_0^t\left(\|\nabla (\phi,u)(\tau)\|_{H^{s-1}}^2+\|\nabla v(\tau)\|_{H^{s}}^2\right)d\tau\leq C\varepsilon_0^2,
		\end{equation}	
		where $C$ is a positive constant independent of time and $\varepsilon_0$ is defined in \eqref{in-data-e0}.	
	\end{prop}
	\begin{proof}
	Let $\gamma_1>0$ be a sufficiently small constant. Taking $(\ref{lem-n1} )+\gamma_1\times(\ref{lem-nx1} )$ yields
	\begin{equation}
		\begin{aligned}
			&\frac{d}{dt}\left\{\frac{1}{2}(\|\phi\|_{H^s}^2+\| u\|_{H^s}^2+\frac{1}{c}\|v\|_{H^s}^2)+\gamma_1\sum_{k=1}^s\int_{\mathbb{R}^3} \nabla^{k-1}u\cdot\nabla^k \phi dx \right\}\\
			&\quad +\|u-v\|_{H^s}^2+\frac{1}{c}\|\nabla v\|_{H^s}^2+\frac{1}{2}\gamma_1\|\nabla \phi\|_{H^{s-1}}^2\\
			&\leq C\delta (\|\nabla \phi\|_{H^{s-1}}^2+\| \nabla u\|_{H^{s-1}}^2+
			\|u-v\|_{H^{s}}^2+\|\nabla v\|_{H^{s-1}}^2)\\
			&\quad+C\gamma_1(\|\nabla u\|_{H^{s-1}}^2+\|u-v\|_{H^{s-1}}^2).
		\end{aligned}
	\end{equation}
	It is obviously from the smallness of $\gamma_1$ that 
	$$
	\frac{1}{2}(\|\phi\|_{H^s}^2+\| u\|_{H^s}^2+\frac{1}{c}\|v\|_{H^s}^2)+\gamma_1\sum_{k=1}^s\int_{\mathbb{R}^3} \nabla^{k-1}u\cdot\nabla^k \phi dx  \sim 
	\|(\phi,u,v)(t)\|_{H^s}^2.
	$$
	Due to the smallness $\delta$ and $\gamma_1$, there exists a positive constant $C$ such that
	\begin{equation}\label{At}
		\frac{d}{dt}\|(\phi,u,v)(t)\|_{H^s}^2
		+C(\|\nabla (\phi,u)(t)\|_{H^{s-1}}^2+\|\nabla v(t)\|_{H^s}^2)\leq 0,
	\end{equation}
	where we have used 
	$$
	\|(u-v)\|_{H^s}^2+\|\nabla v\|_{H^s}^2\geq\|\nabla u\|_{H^{s-1}}^2.
	$$
Integrating \eqref{At} with respect to time from $0$ to $t$
	\begin{equation}
		\|(\phi,u,v)(t)\|_{H^s}^2+C\int_0^t\left(\|\nabla (\phi,u)(\tau)\|_{H^{s-1}}^2+\|\nabla v(\tau)\|_{H^{s}}^2\right)d\tau
		\leq C\varepsilon_0^2.
	\end{equation}	
	This completes the proof of Proposition \ref{pro-est}.
	\end{proof}

	\section{Time decay estimates of the E-NS system}\label{Sec4}
\hspace{2em} In this section, we will investigate the large-time behavior of the E-NS system. More precisely, we first establish the time decay estimates of the linear system by using spectral analysis. Then, by Duhamel's principle, low-high frequency decomposition, and energy method, we derive the time decay estimates of the nonlinear system.
	\subsection{Time decay estimates of the linear equations}
\hspace{2em}	Firstly, we write the system \eqref{3-1}-\eqref{3-21} into the following form
	\begin{equation}\label{Main3}
		\left\{
		\begin{aligned}
			&\partial_t\phi +{\rm{div}}u=f_1, \\
			&\partial_tu+\nabla \phi+u-v=f_2,\\
			&\partial_tv+cv-c\mathcal{J}u-\Delta v=f_3,
		\end{aligned}
		\right.
	\end{equation}
	subject to the initial data
	\begin{equation}\label{in-data2}
		(\phi,u,v)|_{t=0}=(\phi_0,u_0,v_0),\quad \text{div}v_0=0,
	\end{equation}
	and far-field states
	\begin{equation}\label{3-221}
		(\phi,u,v)\rightarrow (0,0,0),\quad \text{as}\quad |x|\rightarrow+\infty.
	\end{equation}
	The linear operator $\mathcal{J}$ is given by 
	$$
	\mathcal{J}=\rm{I}+\nabla(-\Delta)^{-1}\text{div},
	$$
	and the nonlinear terms $f_1, f_2,f_3$ satisfy
	\begin{equation}\label{non-f}
		\begin{aligned}
			f_1=-u\cdot\nabla\phi,~f_2=-u\cdot\nabla u,~~f_3=-\mathcal{J}(v\cdot\nabla v)+\mathcal{J}(c(e^{\phi}-1)(u-v)).
		\end{aligned}
	\end{equation}
	Let $U=(\phi,u,v)^t$, $U_0=(\phi_0,u_0,v_0)^t,$ and $F=(f_1,f_2,f_3)^t$. Then  the system \eqref{Main3}-\eqref{in-data2} can be written as  
	\begin{equation}\label{linearized1}
		\left\{
		\begin{aligned}
			&\partial_tU+\mathcal{L}U=F,\\
			&U|_{t=0}=U_0,
		\end{aligned}
		\right.
	\end{equation}
	where the linear operator $\mathcal{L}$ is given by 
	\begin{equation}
		\mathcal{L}=\left(\begin{array}{ccc}
			0 & \rm{div} & 0  \\
			\nabla &  {\rm{I}}  & -{\rm{I}} \\
			0 &-c\mathcal{J}& c{\rm{I}} -\Delta \\
		\end{array}\right).
	\end{equation}
	In virtue of Duhamel's principle, the solution can be stated as follows 
	\begin{equation}\label{solutionofF}
		\begin{aligned}
			U(t,x)
			&=e^{-t\mathcal{L}}U_0+\int_0^t e^{-(t-\tau)\mathcal{L}}F(\tau)d\tau\\
			&=G*U_0+\int_0^tG(t-\tau)*F(\tau)d\tau.
		\end{aligned}
	\end{equation}
	$G(t,x)$ is the Green function  satisfying
	\begin{equation}
		\left\{
		\begin{aligned}
			&\partial_tG+\mathcal{L}G=0,\\
			&G(0,x)=\delta(x){\rm{I}},
		\end{aligned}
		\right.
	\end{equation}
	where $\delta(x)$ represents the standard Dirac delta function. Consider the linearized system of \eqref{Main3}-\eqref{3-221}
	\begin{equation}\label{Main4}
		\left\{
		\begin{aligned}
			&\partial_t\phi +{\rm{div}}u=0, \\
			&\partial_tu+\nabla \phi+u-v=0,\\
			&\partial_tv+cv-c\mathcal{J}u-\Delta v=0,
		\end{aligned}
		\right.
	\end{equation}
	with the initial data 
	\begin{equation}\label{Main4-1}
		(\phi,u,v)|_{t=0}=(\phi_0,u_0,v_0),
	\end{equation}
	and far-field states
	\begin{equation}\label{32}
		(\phi,u,v)\rightarrow (0,0,0),\quad \text{as}\quad |x|\rightarrow+\infty.
	\end{equation}
	Since $	\text{div}(\mathcal{J}u)=0$, taking div operator in $\eqref{Main4}_3$ yields
	$$
	\partial_t(\text{div}v)+c\text{div}v=\Delta \text{div}v.
	$$
	Solving this equation and noting $\text{div}v_0=0$, we get the following fact for linear system 
	$$
	\text{div}v=0.
	$$
	Therefore, we decompose the solution into two parts
	\begin{equation}\label{UT}
		u(t,x)=-\Lambda^{-1}\nabla \Lambda^{-1}\text{div}u+\Lambda^{-1}\text{curl}(\Lambda^{-1}\text{curl}u),\quad v(t,x)=\Lambda^{-1}\text{curl}(\Lambda^{-1}\text{curl}v),
	\end{equation}
	where $\Lambda^{-1}$ is the pseudodifferential operator defined as \eqref{Lambda}.
	
	Denote 
	\begin{equation}
		w=w(t,x)=\Lambda^{-1}\text{div}u,~~\Psi=\Psi(t,x)=\Lambda^{-1}\text{curl} ~u,\quad\varphi=\varphi(t,x)=\Lambda^{-1}\text{curl}~v,
	\end{equation} 
	and 
	\begin{equation}
		w_0=\Lambda^{-1}\text{div}u_0,~~\Psi_0=\Lambda^{-1}\text{curl} ~u_0,\quad\varphi_0=\Lambda^{-1}\text{curl}~v_0.
	\end{equation} 
	From above, we intend to decompose the linear system \eqref{Main4} into two parts: the divergence-free part and the curl-free part:
	\begin{equation}\label{I}
		{\rm{(I)}}	\left\{
		\begin{aligned}
			&\partial_t\phi +\Lambda w=0, \\
			&\partial_tw-\Lambda \phi+w=0,\\
			&(\phi,w)|_{t=0}=(\phi_0,w_0),
		\end{aligned}
		\right.
	\end{equation}
	and
	\begin{equation}\label{II}
		{\rm{(II)}}		\left\{
		\begin{aligned}
			&\partial_t\Psi +\Psi-\varphi=0, \\
			&\partial_t\varphi+c\varphi-c\Psi+\Lambda^2 \varphi=0,\\
			&(\Psi,\varphi)|_{t=0}=(\Psi_0,\varphi_0).
		\end{aligned}
		\right.
	\end{equation}
	Let $U_1=(\phi,w)^t$ and  $U_{10}=(\phi_0,w_0)^t$. We write the system ${\rm{(I)}}$ into the vector form.
	\begin{equation}\label{linear1}
		\left\{
		\begin{aligned}
			&\partial_tU_1+\mathcal{L}_1U_1=0,\\
			&U_1|_{t=0}=U_{10},
		\end{aligned}
		\right.
	\end{equation}
	where the operator $\mathcal{L}_1$ is introduced as 
	\begin{equation}
		\mathcal{L}_1=\left(\begin{array}{cc}
			0 & \Lambda   \\
			-\Lambda & 1 \\
		\end{array}\right).
	\end{equation}
	Let $G_1(t,x)$ be the Green function of \eqref{linear1}. After taking Fourier transform in $x$, we have
	$$
	\partial_t\hat{G}_1(t,\xi)+B_1\hat{G}_1(t,\xi)=0,\quad \hat{G}_1(0,\xi)={\rm{I}},
	$$
	where $B_1=\mathscr{F}[\mathcal{L}_1]$ is denoted by
	\begin{equation}
		B_1=\left(\begin{array}{cc}
			0 & |\xi| \\
			-|\xi| &1\\
		\end{array}\right).
	\end{equation}
	The eigenvalues of $B_1$ are computed from the determinant
	\begin{equation}
		\det{(\lambda {\rm{I}}+B_1)}=\lambda^2+\lambda+|\xi|^2=0.
	\end{equation}
	A simple computation gives
	\begin{equation}
		\lambda_1=\frac{-1+\sqrt{1-4|\xi|^2}}{2}, \quad  \lambda_2=\frac{-1-\sqrt{1-4|\xi|^2}}{2}.
	\end{equation}
	It is easy to verify that $\hat{G}_1(t,\xi)$ can be expressed as
	\begin{equation}
		\hat{G}_1(t,\xi)=\sum_{k=1}^2e^{\lambda_kt}Q_k(\xi),
	\end{equation}
	where the project operators $Q_{k}(\xi)$ satisfies
	\begin{equation}
		Q_k(\xi)=\prod_{j\neq k}\frac{-B_1-\lambda_{j}{\rm{I}}}{\lambda_k-\lambda_j}.
	\end{equation}
	After a tedious calculation, we find that 
	\begin{equation}
		\hat{G}_1(t,\xi)=
		\left(\begin{array}{cc}
			\frac{\lambda_1e^{\lambda_2t}-\lambda_2e^{\lambda_1t}}{\lambda_1-\lambda_2} & \frac{e^{\lambda_2t}-e^{\lambda_1t}}{\lambda_1-\lambda_2} |\xi| \\
			\frac{e^{\lambda_1t}-e^{\lambda_2t}}{\lambda_1-\lambda_2}|\xi| 
			&\frac{\lambda_1e^{\lambda_1t}-\lambda_2e^{\lambda_2t}}{\lambda_1-\lambda_2}
		\end{array}\right).
	\end{equation}
	Thus, the explicit formulas of $\hat{\phi}$ and $\hat{w}$ are stated as 
	\begin{equation}\label{THZ2}
		\begin{aligned}
			&\hat{\phi}(t,\xi)=\frac{\lambda_1e^{\lambda_2t}-\lambda_2e^{\lambda_1t}}{\lambda_1-\lambda_2}	\hat{\phi}_0(\xi)+\frac{(e^{\lambda_2t}-e^{\lambda_1t})|\xi|}{\lambda_1-\lambda_2}\hat{w}_0(\xi),\\
			&\hat{w}(t,\xi)=\frac{(e^{\lambda_1t}-e^{\lambda_2t})|\xi|}{\lambda_1-\lambda_2} 	\hat{\phi}_0(\xi)+\frac{\lambda_1e^{\lambda_1t}-\lambda_2e^{\lambda_2t}}{\lambda_1-\lambda_2}\hat{w}_0(\xi).
		\end{aligned}	
	\end{equation}
	Let $U_2=(\Psi,\varphi)^t$ and $U_{20}=(\Psi_0,\varphi_0)^t$. Similarly, the system ${\rm{(II)}}$ can be written into the vector form
	\begin{equation}\label{TTU1}
		\left\{
		\begin{aligned}
			&\partial_tU_2+\mathcal{L}_2U_2=0,\\
			&U_2|_{t=0}=U_{20},
		\end{aligned}
		\right.
	\end{equation}
	where the operator $\mathcal{L}_2$ is given by
	\begin{equation}
		\mathcal{L}_2=\left(\begin{array}{cccccc}
			{\rm{I}}&- 	{\rm{I}}\\
			-c	{\rm{I}}& (c+\Lambda^2){\rm{I}}\\
		\end{array}\right).
	\end{equation}
	Let $G_2(t,x)$ be Green function of \eqref{TTU1}. After taking Fourier transform in $x$, we have
	$$
	\partial_t\hat{G}_2(t,\xi)+B_2\hat{G}_2(t,\xi)=0,\quad \hat{G}_2(0,\xi)={\rm{I}},
	$$
	where $B_2=\mathscr{F}[\mathcal{L}_2]$ is given by
	\begin{equation}
		B_2=\left(\begin{array}{cccccc}
			{\rm{I}}&- 	{\rm{I}}\\
			-c	{\rm{I}}& (c+|\xi|^2){\rm{I}}\\
		\end{array}\right).
	\end{equation}
	The eigenvalues of $B_2$ are computed from the determinant
	\begin{equation}
		\det{(\lambda {\rm{I}}+B_2)}=(\lambda^2+(c+1+|\xi|^2)\lambda+|\xi|^2)^3=0,
	\end{equation}
	which gives 
	$$
	\begin{aligned}
		&\lambda_3=\frac{-(c+1+|\xi|^2)+\sqrt{(c+1+|\xi|^2)^2-4|\xi|^2}}{2} ~~~~(triple),\\
		&\lambda_4=\frac{-(c+1+|\xi|^2)-\sqrt{(c+1+|\xi|^2)^2-4|\xi|^2}}{2}~~~~(triple).
	\end{aligned}
	$$
	Thus, $\hat{G}_2(t,\xi)$ can be solved as
	\begin{equation}
		\hat{G}_2(t,\xi)=\sum_{k=3}^4e^{\lambda_kt}Z_k(\xi),
	\end{equation}
	where the project operators $Z_{k}(\xi)$ satisfies
	\begin{equation}
		Z_k(\xi)=\prod_{j\neq k}\frac{-B_2-\lambda_{j}{\rm{I}}}{\lambda_k-\lambda_j}.
	\end{equation}
	After a tedious calculation, it is shown that 
	\begin{equation}
		\hat{G}_2(t,\xi)=
		\left(\begin{array}{cc}
			\frac{(\lambda_3+1)e^{\lambda_4t}-(\lambda_4+1)e^{\lambda_3t}}{\lambda_3-\lambda_4}{\rm{I}}& \frac{e^{\lambda_3t}-e^{\lambda_4t}}{\lambda_3-\lambda_4} {\rm{I}}\\
			\frac{c(e^{\lambda_3t}-e^{\lambda_4t})}{\lambda_3-\lambda_4}{\rm{I}}
			&\frac{(\lambda_3+1)e^{\lambda_3t}-(\lambda_4+1)e^{\lambda_4t}}{\lambda_3-\lambda_4}{\rm{I}}
		\end{array}\right).
	\end{equation}
	Then, the solution is stated below
	\begin{equation}\label{THZ3}
		\begin{aligned}
			& \hat{\Psi}(t,\xi)=\frac{(\lambda_3+1)e^{\lambda_4t}-(\lambda_4+1)e^{\lambda_3t}}{\lambda_3-\lambda_4}\hat{\Psi}_0+\frac{e^{\lambda_3t}-e^{\lambda_4t}}{\lambda_3-\lambda_4}\hat{\varphi}_0,\\
			&\hat{\varphi}(t,\xi)=\frac{c(e^{\lambda_3t}-e^{\lambda_4t})}{\lambda_3-\lambda_4}\hat{\Psi}_0+\frac{(\lambda_3+1)e^{\lambda_3t}-(\lambda_4+1)e^{\lambda_4t}}{\lambda_3-\lambda_4}\hat{\varphi}_0.	
		\end{aligned}	
	\end{equation}
	In terms of \eqref{UT}, \eqref{THZ2}, and \eqref{THZ3}, the solution of $\hat{\phi},\hat{u},\hat{v}$ can be stated as follows:
	\begin{equation}\label{LsU}
		\begin{aligned}
			&\hat{\phi}(t,\xi)=\frac{\lambda_1e^{\lambda_2t}-\lambda_2e^{\lambda_1t}}{\lambda_1-\lambda_2}	\hat{\phi}_0(\xi)-\frac{e^{\lambda_1t}-e^{\lambda_2t}}{\lambda_1-\lambda_2}i\xi\cdot\hat{u}_0(\xi),\\		
			&\hat{u}(t,\xi)=-\frac{e^{\lambda_1t}-e^{\lambda_2t}}{\lambda_1-\lambda_2}i\xi\cdot\hat{\phi}_0(\xi)+\frac{\lambda_1e^{\lambda_1t}-\lambda_2e^{\lambda_2t}}{\lambda_1-\lambda_2}\frac{\xi\xi^t}{|\xi|^2}\hat{u}_0(\xi)\\
			&\quad+\frac{(\lambda_3+1)e^{\lambda_4t}-(\lambda_4+1)e^{\lambda_3t}}{\lambda_3-\lambda_4}\Big({\rm{I}}-\frac{\xi\xi^t}{|\xi|^2}\Big)\hat{u}_0(\xi)+
			\frac{e^{\lambda_3t}-e^{\lambda_4t}}{\lambda_3-\lambda_4}\hat{v}_0(\xi),\\		
			&\hat{v}(t,\xi)=\frac{c(e^{\lambda_3t}-e^{\lambda_4t})}{\lambda_3-\lambda_4}\Big({\rm{I}}-\frac{\xi\xi^t}{|\xi|^2}\Big)\hat{u}_0(\xi)+\frac{(\lambda_3+1)e^{\lambda_3t}-(\lambda_4+1)e^{\lambda_4t}}{\lambda_3-\lambda_4}\hat{v}_0(\xi).
		\end{aligned}
	\end{equation}
	Consequently, the Fourier transform of Green function $G(t,x)=(G_{ik})_{3\times 3}$ is calculated as 
	\begin{equation}\label{G1}
	\begin{aligned}
		&\hat{G}(t,\xi)=
		\left(\begin{array}{ccc}
			\hat{G}_{11} & \hat{G}_{12} &\hat{G}_{13}\\
			\hat{G}_{21} & \hat{G}_{22} &\hat{G}_{23}\\
			\hat{G}_{31} & \hat{G}_{32} &\hat{G}_{33}
		\end{array}\right),
	\end{aligned}
	\end{equation}
	where $\hat{G}_{ij}$ are defined as
		\begin{equation*}
       \begin{aligned}
		&	\hat{G}_{11}=
				\frac{\lambda_1e^{\lambda_2t}-\lambda_2e^{\lambda_1t}}{\lambda_1-\lambda_2},\quad
				\hat{G}_{12}=-i\xi^t \frac{e^{\lambda_1t}-e^{\lambda_2t}}{\lambda_1-\lambda_2},\quad
				\hat{G}_{13}=\hat{G}_{31}=0,\\
			&\hat{G}_{21}=	-i\xi \frac{e^{\lambda_1t}-e^{\lambda_2t}}{\lambda_1-\lambda_2},\quad\hat{G}_{22}= 
				\frac{(\lambda_3+1)e^{\lambda_4t}-(\lambda_4+1)e^{\lambda_3t}}{\lambda_3-\lambda_4}\Big({\rm{I}}-\frac{\xi\xi^t}{|\xi|^2}\Big)+\frac{\lambda_1e^{\lambda_1t}-\lambda_2e^{\lambda_2t}}{\lambda_1-\lambda_2}\frac{\xi\xi^t}{|\xi|^2},\\
		&	\hat{G}_{23}=\frac{e^{\lambda_3t}-e^{\lambda_4t}}{\lambda_3-\lambda_4},\quad	\hat{G}_{32}=\frac{c(e^{\lambda_3t}-e^{\lambda_4t})}{\lambda_3-\lambda_4}\Big({\rm{I}}-\frac{\xi\xi^t}{|\xi|^2}\Big),\quad\hat{G}_{33}=\frac{(\lambda_3+1)e^{\lambda_3t}-(\lambda_4+1)e^{\lambda_4t}}{\lambda_3-\lambda_4}.
		\end{aligned}
	\end{equation*}

	Firstly, by a simple calculation, we investigate the asymptotic behavior of the eigenvalues $\lambda_i (i=1,2,3,4)$ as follows: 
	\begin{lem}\label{THZ4}
		Assume $r_0$ is a sufficiently small positive constant. For low frequency part $|\xi|<r_0$, the eigenvalues satisfy
		\begin{equation}
			\begin{aligned}
				&\lambda_1=-|\xi|^2+O(|\xi|^4),\\
				&\lambda_2=-1+|\xi|^2+O(|\xi|^4),\\ &\lambda_3=-\frac{|\xi|^2}{c+1}+O(|\xi|^4),\\
				&\lambda_4=-(c+1)-\frac{c|\xi|^2}{c+1}+O(|\xi|^4).
			\end{aligned}
		\end{equation}
		For $|\xi|\geq r_0$, there exists a positive constant $R$ such that
		\begin{equation}
			{\rm{Re}}~(\lambda_1,\lambda_2,\lambda_3,\lambda_4)\leq -R.
		\end{equation}
	\end{lem}

Then, we directly obtain the time decay rates of the Green function $G(t,x)$. 
	\begin{prop}\label{lemL1}
		For a given function $f(t,x)$, we have the following decay estimates of the Green function $G(t,x)=(G_{ik})_{3\times 3}$,
		\begin{equation}
			\|\nabla^k G*f\|_{L^2}\leq C(1+t)^{-\frac{3}{4}-\frac{k}{2}}\|f\|_{L^1}+Ce^{-Rt}\|\nabla^k f\|_{L^2}.
		\end{equation}
		In particular, for the low-frequency part $G^\ell(t,x)=\mathcal{K}_1G(t,x)=(G^\ell_{ik})_{3\times 3}$, it holds
		\begin{equation}\label{LowLG}
			\|\nabla^k G^\ell*f\|_{L^2}\leq C(1+t)^{-\frac{3}{4}-\frac{k}{2}}\|f\|_{L^1}.
		\end{equation}
	\end{prop}
	Finally, we derive the lower bounds on the above decay rate of $(\phi,u,v)$ to the linearized problem \eqref{Main4}-\eqref{32} by selecting special initial data.
	\begin{prop}\label{pro-li-d}
		Assume the conditions in Theorem \ref{thm3} hold. The global solution $(\phi,u,v)$ of the linearized problem \eqref{Main4}-\eqref{32} satisfies for sufficiently large-time $t\geq t_0$ that
		\begin{equation}\label{pro-li-d1}
			c_*(1+t)^{-\frac{3}{4}}\leq \|(\phi,u,v)(t)\|_{L^2}\leq C(1+t)^{-\frac{3}{4}},
		\end{equation}
		where $C$  and $c_*$ are positive constants independent of time.
	\end{prop}
	\begin{proof}
		As a result of Proposition \ref{lemL1}, one has 
		\begin{equation}\label{UT1}
		\|(\phi,u,v)(t)\|_{L^2}=\|U\|_{L^2}\leq \|G*U_0\|_{L^2}\leq C\|G\|_{L^2}\|U_0\|_{L^1}\leq C(1+t)^{-\frac{3}{4}}.	
		\end{equation}
		It remains to establish the lower bound of the time decay rates. By \eqref{LsU}, we have  
		\begin{equation}\label{pro-li-d2}
			\begin{aligned}
				\hat{\phi}(t,\xi)=\frac{\lambda_1e^{\lambda_2t}-\lambda_2e^{\lambda_1t}}{\lambda_1-\lambda_2}	\hat{\phi}_0(\xi)-\frac{e^{\lambda_1t}-e^{\lambda_2t}}{\lambda_1-\lambda_2}i\xi\cdot\hat{u}_0(\xi)\triangleq Y_1+Y_2.	
			\end{aligned}
		\end{equation}
		In terms of the assumption \eqref{in-data-optimal}, it is easily seen that
		$$
		\begin{aligned}
			\|Y_1\|_{L^2}^2
			&=\int_{\mathbb{R}^3}\Big|\frac{\lambda_1e^{\lambda_2t}-\lambda_2e^{\lambda_1t}}{\lambda_1-\lambda_2}	\hat{\phi}_0(\xi)\Big|^2d\xi\\
			&\geq\int_{|\xi|<r_0}\Big|\frac{\lambda_1e^{\lambda_2t}-\lambda_2e^{\lambda_1t}}{\lambda_1-\lambda_2}	\hat{\phi}_0(\xi)\Big|^2d\xi\\
			&\geq c_0^2\int_{|\xi|<r_0}\Big|\frac{\lambda_1e^{\lambda_2t}-\lambda_2e^{\lambda_1t}}{\lambda_1-\lambda_2}\Big|^2d\xi\\
			&\geq c_1^2(1+t)^{-\frac{3}{2}},
		\end{aligned}
		$$
		where $c_1$ is a positive constant independent of time. As for $Y_2$, it can be bounded by 
		$$
		\|Y_2\|_{L^2}\leq  C(1+t)^{-\frac{5}{4}}.
		$$
		There we obtain for large-time $t\geq t_0$ that 
		\begin{equation}\label{LT3}
			\|\phi\|_{L^2}=\|\hat{\phi}\|_{L^2}\geq \|Y_1\|_{L^2}-\|Y_2\|_{L^2}
			\geq c_1(1+t)^{-\frac{3}{4}}-C(1+t)^{-\frac{5}{4}}\geq \frac{1}{2}c_1(1+t)^{-\frac{3}{4}}.
		\end{equation}
		As for the velocity $u(t,x)$, we have by \eqref{in-data-optimal} that 
		$$
		\begin{aligned}
			\hat{u}(t,\xi)	
			&=-\frac{e^{\lambda_1t}-e^{\lambda_2t}}{\lambda_1-\lambda_2}i\xi\cdot\hat{\phi}_0(\xi)+\frac{\lambda_1e^{\lambda_1t}-\lambda_2e^{\lambda_2t}}{\lambda_1-\lambda_2}\frac{\xi\xi^t}{|\xi|^2}\hat{u}_0(\xi)
			+\frac{e^{\lambda_3t}-e^{\lambda_4t}}{\lambda_3-\lambda_4}\hat{v}_0(\xi)\\	
			&\triangleq Y_3+Y_4+Y_5.
		\end{aligned}
		$$
		It then follows from Lemma \ref{THZ4} and direct calculations to prove 
		$$
		\|Y_3\|_{L^2}\leq C(1+t)^{-\frac{5}{4}},\quad \|Y_4\|_{L^2}\leq C(1+t)^{-\frac{7}{4}}.
		$$
		According to the assumption \eqref{in-data-optimal}, we can prove that 
		$$
		\begin{aligned}
			\|Y_5\|_{L^2}^2
			&=\int_{\mathbb{R}^3}\Big|\frac{e^{\lambda_3t}-e^{\lambda_4t}}{\lambda_3-\lambda_4}	\hat{v}_0(\xi)\Big|^2d\xi\\
			&\geq\int_{|\xi|<r_0}\Big|\frac{e^{\lambda_3t}-e^{\lambda_4t}}{\lambda_3-\lambda_4}	\hat{v}_0(\xi)\Big|^2d\xi\\
			&\geq c_0^2\int_{|\xi|<r_0}\Big|\frac{e^{\lambda_3t}-e^{\lambda_4t}}{\lambda_3-\lambda_4}\Big|^2d\xi\\
			&\geq c_2^2(1+t)^{-\frac{3}{2}},
		\end{aligned}
		$$
		where $c_2$ represents a positive constant independent of time. Thus, we deduce for large time $t\geq t_0$ that
		\begin{equation}\label{pro-li-d8}
			\begin{aligned}
				\|u(t)\|_{L^2}^2&\geq \int_{\mathbb{R}^3} |Y_3+Y_4+Y_5|^2d\xi\\
				&\geq \frac{1}{2}\|Y_5\|_{L^2}^2-2\|Y_4\|_{L^2}^2-2\|Y_3\|_{L^2}^2\\
				&\geq \frac{1}{2}c_2^2(1+t)^{-\frac{3}{2}}-C(1+t)^{-\frac{7}{2}}-C(1+t)^{-\frac{5}{2}}\\
				&\geq \frac{1}{4}c_2^2(1+t)^{-\frac{3}{2}},
			\end{aligned}
		\end{equation}
		which gives 
		\begin{equation}\label{LT2}
			\|u(t)\|_{L^2}\geq \frac{1}{2}c_2(1+t)^{-\frac{3}{4}}.
		\end{equation}
		In the same way, we can  obtain the lower bound of the time decay rate of $v(t)$
		\begin{equation}\label{LT1}
			\|v(t)\|_{L^2}\geq \frac{1}{2}c_3(1+t)^{-\frac{3}{4}},
		\end{equation}
		where $c_3>0$ is a constant independent of time.  Let $c_*=\frac{1}{2}\min\{c_1,c_2,c_3 \}$. Combining  \eqref{LT3}, \eqref{LT2}, \eqref{LT1}, and \eqref{UT1} yields  
		\begin{equation}
			c_*(1+t)^{-\frac{3}{4}}\leq \|(\phi,u,v)(t)\|_{L^2}\leq C(1+t)^{-\frac{3}{4}},
		\end{equation}
		where $C$  and $c_*$ are positive constants independent of time. Thus we have completed the proof.
	\end{proof}
	\subsection{Time decay estimates of the nonlinear equations}\label{subsec5-non-decay}
\hspace{2em}	In this subsection, we intend to obtain the upper bound of the optimal decay rates stated in \eqref{Upper}. To this end, we introduce the energy for any $0\leq j\leq s$,
	\begin{equation}\label{eskt}
		\mathcal{E}_j^s(t)=\|\nabla^j \phi (t)\|_{H^{s-j}}^2+\|\nabla^j u(t)\|_{H^{s-j}}^2+\|\nabla^j v(t)\|_{H^{s-j}}^2,
	\end{equation}
	and the time-weighted energy functional
	\begin{equation}\label{mskt}
		M(t)=\sup_{0<\tau\leq t}\left\{(1+\tau)^{\frac{3}{4}}\|(\phi,u,v)(\tau)\|_{H^{s}}\right\}.
	\end{equation}
	It is evident that 
	\begin{equation}\label{THZ1}
		\|(\phi,u,v)(t)\|_{L^2}\leq (1+t)^{-\frac{3}{4}}M(t).
	\end{equation}
	The next goal is to show that $M(t)$ has a uniform upper bound independent of time. 
	
	\begin{lem}\label{lemA1}
		Assume the conditions in Theorem \ref{thm1} hold. There exists a positive constant $C$ depended on $ \mathcal{I}_0$ and $\varepsilon_0$ such that
		\begin{equation}\label{lemA1-1}
			\|(\phi,u,v)(t)\|_{H^s}\leq C(1+t)^{-\frac{3}{4}}.
		\end{equation}
	\end{lem}
	\begin{proof}
		In terms of \eqref{At} and the definition of $\mathcal{E}_0^s(t)$, one has
		\begin{equation*}
			\frac{d}{dt}\mathcal{E}^s_0(t)+C(\|\nabla (\phi,u) (t)\|_{H^{s-1}}^2+\|\nabla v(t)\|_{H^s}^2)\leq 0.
		\end{equation*}
	Applying the fact in Lemma \ref{lema6} that $\|(\phi^{h},u^{h},v^{h})(t)\|_{L^2}\leq C \|\nabla(\phi,u,v)(t)\|_{L^2}$, we can prove there exists a positive constant $C_1$ such that
		\begin{equation}\label{lem-hs1}
			\frac{d}{dt}\mathcal{E}^s_0(t)+C_1\mathcal{E}^s_0(t)\leq C\|(\phi^{\ell},u^{\ell},v^{\ell})(t)\|^2_{L^2}.
		\end{equation}
	Utilizing Duhamel's principle, for the low frequency part $U^\ell(t,x)=(\phi^{\ell},u^{\ell},v^{\ell})^t$, one has
		\begin{equation}\label{lem-hs2}
			U^\ell(t,x)=G^\ell*U_0+\int_0^tG^\ell(t-\tau)*F(\tau)d\tau.
		\end{equation}
		The nonlinear term $F=(f_1,f_2,f_3)^t$ satisfies
		$$
		f_1=-u\cdot\nabla\phi,~f_2=-u\cdot\nabla u,~~f_3=-\mathcal{J}(v\cdot\nabla v)+\mathcal{J}(c(e^{\phi}-1)(u-v)).
		$$
It is easy to get from \eqref{G1} that $G_{13}=G_{31}=0$. Then we have  
			\begin{align*}
				&\int_0^t \|G^\ell(t-\tau)*F(\tau)\|_{L^2}d\tau \\
				&\leq \int_0^t\Big(\|(G_{11}^\ell+G_{21}^\ell)(t-\tau)*f_1(\tau)\|_{L^2} +\|(G_{12}^\ell+G_{22}^\ell+G_{32}^\ell)(t-\tau)*f_2(\tau)\|_{L^2}\Big)d\tau\\
				&\quad+\int_0^t \|(G_{23}^\ell+G^\ell _{33})(t-\tau)*f_3(\tau)\|_{L^2}d\tau\\
				&\leq C\int_0^t(1+t-\tau)^{-\frac{3}{4}}\Big(\|(u\cdot \nabla\phi)(\tau)\|_{L^1}+\|(u\cdot\nabla u)(\tau)\|_{L^1}\\
				&\quad+\|(v\cdot\nabla v)(\tau)\|_{L^1}+\|((e^\phi-1)(u-v))(\tau)\|_{L^1}\Big)d\tau\\
				& \leq CM(t)\int_0^t(1+t-\tau)^{-\frac{3}{4}}(1+\tau)^{-\frac{3}{4}}\Big(\| \nabla(\phi,u,v)(\tau)\|_{L^2}+\|(u-v)(\tau)\|_{L^2}\Big)d\tau,
			\end{align*}
		where we have used the fact that the Fourier transform of the operator $\hat{\mathcal{J}}={\rm{I}}-\frac{\xi\xi^t}{|\xi|^2}$ is bounded. 
		
	Using the Parseval's equality, H$\ddot{\text{o}}$lder's inequality, we have
		\begin{equation}\label{lem-hs3}
			\begin{aligned}
				&\|(\phi^{\ell},u^{\ell},v^{\ell})(t)\|_{L^2}\\
				&\leq \|G^\ell*U_0\|_{L^2}+\int_0^t\|G^\ell(t-\tau)*F(\tau)\|_{L^2}d\tau\\
				&\leq C(1+t)^{-\frac{3}{4}}\|U_0\|_{L^1}+CM(t)\int_0^t(1+t-\tau)^{-\frac{3}{4}}(1+\tau)^{-\frac{3}{4}}\Big(\|\nabla (\phi,u,v)(\tau)\|_{L^2}+\|(u-v)(\tau)\|_{L^2}\Big)d\tau\\
				&\leq C \mathcal{I}_0(1+t)^{-\frac{3}{4}}+CM(t)\Big(\int_0^t(1+t-\tau)^{-\frac{3}{2}}(1+\tau)^{-\frac{3}{2}}d\tau\Big)^{\frac{1}{2}}\Big(\int_0^t(\|\nabla (\phi,u,v)(\tau)\|_{L^2}^2+\|(u-v)(\tau)\|_{L^2}^2)d\tau\Big)^{\frac{1}{2}}\\
				&\leq C(1+t)^{-\frac{3}{4}}( \mathcal{I}_0+\varepsilon_0M(t)),
			\end{aligned}
		\end{equation}
		where the last step used the fact that 
		\begin{equation}\label{D1}
			\int_0^t(\|\nabla (\phi,u,v)(\tau)\|_{L^2}^2+\|(u-v)(\tau)\|_{L^2}^2)d\tau\leq 
			C\varepsilon_0^2.
		\end{equation}
		Substituting \eqref{lem-hs3} into \eqref{lem-hs1}, one has
		\begin{equation*}
			\frac{d}{dt}\mathcal{E}^s_0(t)+C_1\mathcal{E}^s_0(t)\leq C(1+t)^{-\frac{3}{2}}\left( \mathcal{I}_0+\varepsilon_0M(t)\right)^2.
		\end{equation*}
		Applying Gr$\ddot{\text{o}}$nwall's inequality yields that
		\begin{equation*}
			\begin{aligned}
				\mathcal{E}^s_0(t)
				\leq C(1+t)^{-\frac{3}{2}}\left(\varepsilon_0+ \mathcal{I}_0+\varepsilon_0M(t)\right)^2.
			\end{aligned}
		\end{equation*}
	Then it follows from above and \eqref{mskt} to prove 
		\begin{equation}
			M(t)\leq C\left(\varepsilon_0+ \mathcal{I}_0+\varepsilon_0M(t)\right).
		\end{equation}
		Since $\varepsilon_0$ is small enough and $ \mathcal{I}_0$ is bounded, it gives rise to
		\begin{equation}\label{lem-hs5}
			M(t)\leq C(\varepsilon_0+ \mathcal{I}_0).
		\end{equation}
	As a consequence, one has
		\begin{equation}
			\|(\phi,u,v)(t)\|_{H^s}\leq C(\varepsilon_0+ \mathcal{I}_0)(1+t)^{-\frac{3}{4}}\leq C(1+t)^{-\frac{3}{4}}.
		\end{equation}
		This completes the proof of Lemma \ref{lemA1}.
	\end{proof}
	
Noting that the decay rates of high-order derivatives are slow. Thus we intend to improve the time decay rate of the derivatives for $ 1\leq j\leq s$.
	\begin{lem}\label{lem-hk}
		Assume the conditions in Theorem \ref{thm1} hold. There exists a positive constant $C$ depended on $\varepsilon_0$ and $ \mathcal{I}_0$ such that for $1\leq j\leq s$, it holds
		\begin{equation}\label{lem-hk0}
			\|\nabla^j(\phi,u,v)(t)\|_{H^{s-j}}\leq C(1+t)^{-\frac{3}{4}-\frac{j}{2}}.
		\end{equation}
	\end{lem}
	\begin{proof}
		According to \eqref{A1} and the smallness of $\delta$, we have 
		\begin{equation}\label{A4}
			\begin{aligned}
				&\frac{1}{2}\frac{d}{dt}(\|\nabla^k \phi \|_{L^2}^2+\|\nabla^k u\|_{L^2}^2+\frac{1}{c}\|\nabla^kv\|_{L^2}^2)
				+\frac{3}{4}\|\nabla^k(u-v)\|_{L^2}^2+\frac{1}{c}\|\nabla\nabla^kv\|_{L^2}^2\\
				&\leq C\delta (\|\nabla^k \phi \|_{L^2}^2+\|\nabla^ku\|_{L^2}^2+\|\nabla^kv\|_{L^2}^2).
			\end{aligned}
		\end{equation}
	Similar to Lemma \ref{lem-nx}, applying the operator $\nabla^{k-1} \mathcal{K}_{\infty}$ to $\eqref{3-1}_2$, multiplying the resulting equation by $\nabla^{k} \phi^h$, and integrating over $\mathbb{R}^3$, using Lemma \ref{lema6}, we obtain
		\begin{equation}\label{0102}
			\begin{aligned}
				&\|\nabla^k \phi^h\|_{L^2}^2+\frac{d}{dt}\int_{\mathbb{R}^3} \nabla^{k-1}u^h\cdot\nabla^k \phi^hdx\\
				&=-\int_{\mathbb{R}^3}\nabla^{k-1}u^h\cdot\nabla^k\text{div}u^hdx-\int_{\mathbb{R}^3}\nabla^{k-1}(u^h-v^h)\cdot\nabla^k \phi ^hdx\\
				&\quad -\int_{\mathbb{R}^3}\nabla^k\mathcal{K}_\infty(u\cdot\nabla \phi)\cdot \nabla^{k-1}u^hdx-\int_{\mathbb{R}^3}\nabla^{k-1}\mathcal{K}_\infty(u\cdot\nabla u)\cdot \nabla^k \phi ^hdx.
			\end{aligned}
		\end{equation}
		The first term and the second term of the right-hand side of \eqref{0102} can be estimated as follows
		\begin{equation}
			\begin{aligned}
				&\Big|\int_{\mathbb{R}^3}\nabla^{k-1}u^h\cdot\nabla^k\text{div}u^hdx\Big|+\Big|\int_{\mathbb{R}^3}\nabla^{k-1}(u^h-v^h)\cdot\nabla^k \phi ^hdx\Big|\\
				&\leq \|\nabla^ku^h\|_{L^2}^2+\|\nabla^{k-1}(u^h-v^h)\|_{L^2}\|\nabla^{k}\phi^h\|_{L^2}\\
				&\leq C\|\nabla^ku\|_{L^2}^2+C\|\nabla^{k}(u-v)\|_{L^2}\|\nabla^{k}\phi^h\|_{L^2}\\
				&\leq C(\|\nabla^ku\|_{L^2}^2+\|\nabla^{k}(u-v)\|_{L^2}^2)+\frac{1}{2}\|\nabla^{k}\phi^h\|_{L^2}^2.
			\end{aligned}
		\end{equation}
		By Young's inequality,  it holds for $2\leq k\leq s$ that 
		\begin{equation}\label{TT1}
			\begin{aligned}
				&	\Big|\int_{\mathbb{R}^3}\nabla^k(u\cdot\nabla \phi)\cdot \nabla^{k-1}udx\Big|\\
				&\leq \Big|\int_{\mathbb{R}^3}\nabla^{k-1}(u\cdot\nabla \phi)\cdot \nabla^{k-1}\text{div}udx\Big|\\
				&\leq \Big|\int_{\mathbb{R}^3}(\nabla^{k-1}(u\cdot\nabla \phi)-u\cdot\nabla^{k-1}\nabla \phi)\cdot \nabla^{k-1}\text{div}udx\Big|+\Big|\int_{\mathbb{R}^3}(u\cdot\nabla^{k-1}\nabla \phi)\cdot \nabla^{k-1}\text{div}udx\Big|\\
				&\leq \big\|[\nabla^{k-1},u]\nabla\phi\big\|_{L^2}\|\nabla^ku\|_{L^2}+C\|u\|_{L^\infty}\|\nabla^k \phi \|_{L^2}\|\nabla^ku\|_{L^2}\\
				&\leq C(\|\nabla u\|_{L^3}\|\nabla^{k-1}\phi\|_{L^6}+\|\nabla^{k-1}u\|_{L^6}\|\nabla \phi\|_{L^3})\|\nabla^{k}u\|_{L^2}+C\delta\|\nabla^k \phi \|_{L^2}\|\nabla^ku\|_{L^2}\\
				&\leq C\delta(\|\nabla^k \phi \|_{L^2}^2+\|\nabla^{k}u\|_{L^2}^2).
			\end{aligned}
		\end{equation}
		After a  direct calculation, it holds for $k=1$ that 
		\begin{equation}\label{TT2}
			\begin{aligned}
				&	\Big|\int_{\mathbb{R}^3}\nabla (u\cdot\nabla \phi)\cdot udx\Big|\leq C\|u\|_{L^\infty}
				\|\nabla\phi\|_{L^2}\|\text{div} u\|_{L^2}\leq C\delta(\|\nabla\phi\|_{L^2}^2+\|\nabla u\|_{L^2}^2).
			\end{aligned}
		\end{equation}
		Combining \eqref{TT1} and \eqref{TT2} yields the following estimates for $1\leq k\leq s$,
		\begin{equation}
			\Big|\int_{\mathbb{R}^3}\nabla^k(u\cdot\nabla \phi)\cdot \nabla^{k-1}udx\Big| \leq C\delta(\|\nabla^k \phi \|_{L^2}^2+\|\nabla^{k}u\|_{L^2}^2).
		\end{equation}
		Thus, we have 
		$$
		\begin{aligned}
			&\Big|\int_{\mathbb{R}^3}\nabla^k\mathcal{K}_\infty(u\cdot\nabla \phi)\cdot \nabla^{k-1}udx\Big|\\
			&\leq \Big|\int_{\mathbb{R}^3}\nabla^k(u\cdot\nabla \phi)\cdot \nabla^{k-1}udx\Big|+	\Big|\int_{\mathbb{R}^3}\nabla^k\mathcal{K}_1(u\cdot\nabla \phi)\cdot \nabla^{k-1}u dx\Big|\\
			&\leq C\delta(\|\nabla^k \phi \|_{L^2}^2+\|\nabla^{k}u\|_{L^2}^2).
		\end{aligned}
		$$
		In a similar way, we also have 
		$$
		\Big|\int_{\mathbb{R}^3}\nabla^{k-1}\mathcal{K}_\infty(u\cdot\nabla u)\cdot \nabla^k \phi ^hdx\Big|
		\leq C\delta(\|\nabla^k \phi \|_{L^2}^2+\|\nabla^{k}u\|_{L^2}^2).
		$$
		Combining the above estimates yields
		\begin{equation}\label{A5}
			\begin{aligned}
				&\frac{1}{2}\|\nabla^k \phi^h\|_{L^2}^2+\frac{d}{dt}\int_{\mathbb{R}^3} \nabla^{k-1}u^h\cdot\nabla^k \phi^hdx\\
				&\leq C(\|\nabla^ku\|_{L^2}^2+\|\nabla^{k}(u-v)\|_{L^2}^2)+C\delta(\|\nabla^k \phi \|_{L^2}^2+\|\nabla^{k}u\|_{L^2}^2).
			\end{aligned}
		\end{equation}
		Choosing the positive constant $\gamma_2$ sufficiently small, then taking $\eqref{A4}+\gamma_2\times(\ref{A5})$ yields
		$$
		\begin{aligned}
			&	\frac{d}{dt}\left\{\frac{1}{2}\big(\|\nabla^k \phi \|_{L^2}^2+\|\nabla^k u\|_{L^2}^2+\frac{1}{c}\|\nabla^kv\|_{L^2}^2\big)+\gamma_2\int_{\mathbb{R}^3} \nabla^{k-1}u^h\cdot\nabla^k \phi^hdx\right\}\\
			&\quad+\frac{3}{4}\|\nabla^k(u-v)\|_{L^2}^2+\frac{1}{c}\|\nabla\nabla^kv\|_{L^2}^2+\frac{1}{2}\gamma_2\|\nabla^k \phi^h\|_{L^2}^2\\
			&\leq  C\delta (\|\nabla^k \phi \|_{L^2}^2+\|\nabla^ku\|_{L^2}^2+\|\nabla^kv\|_{L^2}^2)\\
			&\quad+C\gamma_2(\|\nabla^ku\|_{L^2}^2+\|\nabla^{k}(u-v)\|_{L^2}^2)+C\gamma_2\delta(\|\nabla^k \phi \|_{L^2}^2+\|\nabla^{k}u\|_{L^2}^2).
		\end{aligned}
		$$
		Since $\gamma_2$ and $\delta$ are sufficiently small,  one has
		\begin{equation}\label{lem-hk4}
			\begin{aligned}
				&	\frac{d}{dt}\left\{\frac{1}{2}\big(\|\nabla^k \phi \|_{L^2}^2+\|\nabla^k u\|_{L^2}^2+\frac{1}{c}\|\nabla^kv\|_{L^2}^2\big)+\gamma_2\int_{\mathbb{R}^3} \nabla^{k-1}u^h\cdot\nabla^k \phi^hdx\right\}\\
				&\quad+\frac{1}{2}\|\nabla^k(u-v)\|_{L^2}^2+\frac{1}{c}\|\nabla\nabla^kv\|_{L^2}^2+\frac{1}{2}\gamma_2\|\nabla^k \phi^h\|_{L^2}^2\\
				&\leq C(\delta+\gamma_2+\gamma_2\delta)(\|\nabla^k \phi\|_{L^2}^2+\|\nabla^k u\|_{L^2}^2+\|\nabla^k v\|_{L^2}^2).
			\end{aligned}
		\end{equation}
		By Lemma \ref{lema6} and the smanllness of $\gamma_2$, we are able to prove that 	
		\begin{equation*}
			\begin{aligned}
				&\frac{1}{2}(\|\nabla^k \phi\|_{L^2}^2+\|\nabla^k u\|_{L^2}^2+\frac{1}{c}\|\nabla^k v\|_{L^2}^2)+\gamma_2\int_{\mathbb{R}^3} \nabla^{k-1}u^h\cdot\nabla^k \phi^hdx\\
				&\sim (\|\nabla^k \phi\|_{L^2}^2+\|\nabla^k u\|_{L^2}^2+\|\nabla^k v\|_{L^2}^2).
			\end{aligned}
		\end{equation*}
		It then follows from the above estimates to prove that there exists a constant $C_2$ such that 
		\begin{equation}\label{lem-hk5}
			\begin{aligned}
				&\frac{d}{dt}(\|\nabla^k \phi\|_{L^2}^2+\|\nabla^k u\|_{L^2}^2+\|\nabla^k v\|_{L^2}^2)\\
				&+C_2(\|\nabla^{k} \phi\|_{L^2}^2+\|\nabla^{k} u\|_{L^2}^2+\|\nabla^{k} v\|_{L^2}^2)\\
				&\leq C(\|\nabla^{k}\phi^\ell\|_{L^2}^2+\|\nabla^{k}u^\ell\|_{L^2}^2+\|\nabla^{k}v^\ell\|_{L^2}^2).
			\end{aligned}
		\end{equation}
		Summing up \eqref{lem-hk5} with respect to $k$ from $j$ to $s$,  then using the definition of $\mathcal{E}_j^s(t)$ yields
		\begin{equation}\label{lem-hk-e}
			\frac{d}{dt}\mathcal{E}_{j}^s(t)+C_2\mathcal{E}_{j}^s(t)\leq C\|\nabla^{j}(\phi^{\ell},u^{\ell},v^{\ell})\|_{L^2}^2.
		\end{equation}
		
		We are now in a position to establish the $L^2$ time decay estimates of $\nabla^j(\phi^\ell,u^\ell,v^\ell)$ with $1\leq j \leq  s$. First, similar to the proof of Lemma \ref{lemA1}, we also have
		\begin{equation}\label{lem-hk6}
			\begin{aligned}
			&	\|\nabla(\phi^{\ell},u^{\ell},v^{\ell})(t)\|_{L^2}\\
			&\leq C(1+t)^{-\frac{5}{4}}\|U_0\|_{L^1}
				+C\int_0^{\frac{t}{2}}(1+t-\tau)^{-\frac{5}{4}}
				\Big(\|(u\cdot\nabla\phi)(\tau)\|_{L^1}+\|(u\cdot\nabla u)(\tau)\|_{L^1}\\
				&\quad+\|(v\cdot\nabla v)(\tau)\|_{L^1}+\|((e^\phi-1)(u-v))(\tau)\|_{L^1}\Big)d\tau\\
				&\quad+C\int_{\frac{t}{2}}^t(1+t-\tau)^{-\frac{3}{4}}
				\Big(\|\nabla(u\cdot\nabla\phi)(\tau)\|_{L^1}+\|\nabla(u\cdot\nabla u)(\tau)\|_{L^1}\\
				&\quad+\|\nabla(v\cdot\nabla v)(\tau)\|_{L^1}+\|\nabla((e^\phi-1)(u-v))(\tau)\|_{L^1}\Big)d\tau.
			\end{aligned}
		\end{equation}
		It follows from H\"{o}lder's inequality and \eqref{lemA1-1} that
		\begin{equation}\label{lem-hk7}
			\begin{aligned}
			&\|(u\cdot\nabla\phi)(\tau)\|_{L^1}+\|(u\cdot\nabla u)(\tau)\|_{L^1}+\|(v\cdot\nabla v)(\tau)\|_{L^1}+\|((e^\phi-1)(u-v))(\tau)\|_{L^1}\\
			&\leq \|u\|_{L^2}\|\nabla \phi\|_{L^2}+\|u\|_{L^2}\|\nabla u\|_{L^2}+\|v\|_{L^2}\|\nabla v\|_{L^2}+C\|\phi\|_{L^2}\| u-v\|_{L^2}\\
				&\leq C(1+\tau)^{-\frac{3}{2}}.
			\end{aligned}
		\end{equation}
		Similarly, we also have
		\begin{equation}\label{lem-hk8}
	\|\nabla(u\cdot\nabla\phi)(\tau)\|_{L^1}+\|\nabla(u\cdot\nabla u)(\tau)\|_{L^1}+\|\nabla(v\cdot\nabla v)(\tau)\|_{L^1}+\|\nabla((e^\phi-1)(u-v))(\tau)\|_{L^1}\leq C(1+\tau)^{-\frac{3}{2}}.
		\end{equation}
		Substituting \eqref{lem-hk7}, \eqref{lem-hk8} into \eqref{lem-hk6} yields
			\begin{align*}
				\|\nabla(\phi^{\ell},u^{\ell},v^{\ell})(t)\|_{L^2}&\leq C(1+t)^{-\frac{5}{4}}+C\int_0^{\frac{t}{2}}(1+t-\tau)^{-\frac{5}{4}}(1+\tau)^{-\frac{3}{2}}d\tau\\
				&\quad+C\int_{\frac{t}{2}}^t(1+t-\tau)^{-\frac{3}{4}}(1+\tau)^{-\frac{3}{2}}d\tau\\
				&\leq C(1+t)^{-\frac{5}{4}}.
			\end{align*}
	Then we deduce from \eqref{lem-hk-e} that 
		\begin{equation*}
			\frac{d}{dt}\mathcal{E}_{1}^s(t)+C_2\mathcal{E}_{1}^s(t)\leq C(1+t)^{-\frac{5}{2}}.
		\end{equation*}
		It is easy to check that 
		\begin{equation*}
			\mathcal{E}_1^s(t)\leq C(1+t)^{-\frac{5}{2}},
		\end{equation*}
		which gives the following results for $s\geq 3$
		\begin{equation}
			\|\nabla (\phi,u,v)(t)\|_{H^{s-1}}\leq C(1+t)^{-\frac{5}{4}}.
		\end{equation}
		We proceed to show the time decay rates of the second-order derivatives of the solution.
		\begin{equation}\label{lem-hk9}
			\begin{aligned}
				&	\|\nabla^2(\phi^{\ell},u^{\ell},v^{\ell})(t)\|_{L^2}\\
			&\leq C(1+t)^{-\frac{7}{4}}\|U_0\|_{L^1}
			+C\int_0^{\frac{t}{2}}(1+t-\tau)^{-\frac{7}{4}}
			\Big(\|(u\cdot\nabla\phi)(\tau)\|_{L^1}+\|(u\cdot\nabla u)(\tau)\|_{L^1}\\
			&\quad+\|(v\cdot\nabla v)(\tau)\|_{L^1}+\|((e^\phi-1)(u-v))(\tau)\|_{L^1}\Big)d\tau\\
			&\quad+C\int_{\frac{t}{2}}^t(1+t-\tau)^{-\frac{5}{4}}
			\Big(\|\nabla(u\cdot\nabla\phi)(\tau)\|_{L^1}+\|\nabla(u\cdot\nabla u)(\tau)\|_{L^1}\\
			&\quad+\|\nabla(v\cdot\nabla v)(\tau)\|_{L^1}+\|\nabla((e^\phi-1)(u-v))(\tau)\|_{L^1}\Big)d\tau.
			\end{aligned}
		\end{equation}
		Note that at this moment, we have
		\begin{equation}\label{010301}
			\|(\phi,u,v)(t)\|_{L^2}\leq C(1+t)^{-\frac{3}{4}},\quad \|\nabla(\phi,u,v)(t)\|_{H^{s-1}}\leq C(1+t)^{-\frac{5}{4}}.
		\end{equation}
	Taking $\eqref{3-1}_2$-$\eqref{3-1}_3$, we derive a new equation of $(u-v) $
		\begin{equation}\label{THZ7}
			\partial_t(u-v)+(1+c)(u-v)=-u\cdot\nabla u-\nabla \phi+v\cdot\nabla v+\nabla P-\Delta v-c(e^\phi-1) (u-v).
		\end{equation}
		Taking inner product by \eqref{THZ7} with $(u-v)$ yields
		\begin{equation}\label{THZ8}
			\begin{aligned}
				&\frac{1}{2}\frac{d}{dt}\|u-v\|_{L^2}^2+(1+c)\|u-v\|_{L^2}^2\\
				&=\int_{\mathbb{R}^3}(-u\cdot\nabla u-\nabla \phi+v\cdot\nabla v+\nabla P-\Delta v-c(e^\phi-1) (u-v))\cdot (u-v)dx.
			\end{aligned}
		\end{equation} 
		The right-hand side of \eqref{THZ8} is bounded by 
		$$
		\begin{aligned}
			&\Big|\int_{\mathbb{R}^3}(-u\cdot\nabla u-\nabla \phi+v\cdot\nabla v+\nabla P-\Delta v-c(e^\phi-1) (u-v))\cdot (u-v)dx\Big| \\
			&\leq  (\|u\|_{L^\infty}\|\nabla u\|_{L^2}+\|\nabla \phi\|_{L^2}+\|v\|_{L^\infty}\|\nabla v\|_{L^2}
			+\|\nabla P\|_{L^2}+\|\Delta v\|_{L^2}+c\|e^\phi -1\|_{L^\infty}\|u-v\|_{L^2})\|u-v\|_{L^2}\\
			&\leq C (\|\nabla u\|_{H^1}\|\nabla u\|_{L^2}+\|\nabla\phi\|_{L^2}+\|\nabla v\|_{H^1}\|\nabla v\|_{L^2}+\|\nabla P\|_{L^2}+\|\Delta v\|_{L^2}+\|\nabla \phi\|_{H^1}\|u-v\|_{L^2})\|u-v\|_{L^2}\\
			&\leq C(1+t)^{-\frac{5}{2}}+C\delta \|u-v\|_{L^2}^2.
		\end{aligned}
		$$
		Thus we have 
		\begin{equation}\label{THZ10}
			\frac{d}{dt}\|u-v\|_{L^2}^2+(1+c)\|u-v\|_{L^2}^2\leq C(1+t)^{-\frac{5}{2}}.
		\end{equation}
		Then applying Gr$\ddot{\text{o}}$nwall's inequality yields 
		\begin{equation}\label{THZ11}
			\|(u-v)(t)\|_{L^2}\leq C(1+t)^{-\frac{5}{4}}.
		\end{equation}
		It may be concluded from \eqref{010301} and \eqref{THZ11} that
		\begin{equation}\label{lem-hk10}
			\begin{aligned}
				&\|(u\cdot\nabla\phi)(\tau)\|_{L^1}+\|(u\cdot\nabla u)(\tau)\|_{L^1}+\|(v\cdot\nabla v)(\tau)\|_{L^1}+\|((e^\phi-1)(u-v))(\tau)\|_{L^1}\\
				&\leq \|u\|_{L^2}\|\nabla \phi\|_{L^2}+\|u\|_{L^2}\|\nabla u\|_{L^2}+\|v\|_{L^2}\|\nabla v\|_{L^2}+C\|\phi\|_{L^2}\| u-v\|_{L^2}\\
				&\leq C(1+\tau)^{-2},
			\end{aligned}
		\end{equation}
		and
		\begin{equation}\label{lem-hk11}
			\|\nabla(u\cdot\nabla\phi)(\tau)\|_{L^1}+\|\nabla(u\cdot\nabla u)(\tau)\|_{L^1}+\|\nabla(v\cdot\nabla v)(\tau)\|_{L^1}+\|\nabla((e^\phi-1)(u-v))(\tau)\|_{L^1}\leq C(1+\tau)^{-2}.
		\end{equation}
		Substituting \eqref{lem-hk10} and \eqref{lem-hk11} into \eqref{lem-hk9} yields
		\begin{equation*}
			\begin{aligned}
				\|\nabla^2(\phi^{\ell},u^{\ell},v^{\ell})(t)\|_{L^2}&\leq C(1+t)^{-\frac{7}{4}}+C\int_0^{\frac{t}{2}}(1+t-\tau)^{-\frac{7}{4}}(1+\tau)^{-2}d\tau\\
				&\quad+C\int_{\frac{t}{2}}^t(1+t-\tau)^{-\frac{5}{4}}(1+\tau)^{-2}d\tau\\
				&\leq C(1+t)^{-\frac{7}{4}}.
			\end{aligned}
		\end{equation*}
		With the help of \eqref{lem-hk-e} and choosing $j=2$, we obtain
		\begin{equation*}
			\frac{d}{dt}\mathcal{E}_{2}^s(t)+C_4\mathcal{E}_{2}^s(t)\leq C(1+t)^{-\frac{7}{2}}.
		\end{equation*}
		Applying Gr$\ddot{\text{o}}$nwall's inequality and using the definition of $\mathcal{E}_{2}^s(t)$ yield  for $s\geq 3$ that 
		\begin{equation*}
			\|\nabla^2(\phi,u,v)(t)\|_{H^{s-2}}\leq C(1+t)^{-\frac{7}{4}}.
		\end{equation*}
		We can proceed analogously to the proof of high-order derivatives of the solution. Therefore we obtain \eqref{lem-hk0} and complete the proof of this lemma.
	\end{proof}
	\section{The proof of Theorems \ref{thm1}-\ref{thm3}}\label{Sec5}
	\hspace{2em} Since we have established the uniform estimates in Section \ref{Sec3} and time decay rates in Section \ref{Sec4}, the next goal is to complete the proof of Theorems \ref{thm1}-\ref{thm3}.
	
	\vskip4mm
	
	\noindent{\it\textbf{Proof of Theorem \ref{thm1}.}\ }
	According to Proposition \ref{pro-est}, we are able to prove
	\begin{equation}
		\|(\phi,u,v)(t)\|_{H^s}^2+C\int_0^t\left(\|\nabla (\phi,u)(\tau)\|_{H^{s-1}}^2+\|\nabla v(\tau)\|_{H^{s}}^2\right)d\tau\leq C\varepsilon_0^2.
	\end{equation}
	Due to the smallness of the initial data $\varepsilon_0$, we can choose $\varepsilon_0$ sufficiently small such that $C\varepsilon_0^2 \leq \frac{1}{4}\delta^2$, which closes the \emph{a priori} assumption \eqref{a priori est}. Then, based on the continuous argument, the global existence of solution $(\phi,u,v)$ and the estimate \eqref{main-est} are obtained.
	
	Concerning the large-time behavior of the solution $(\phi,u,v)$, we conclude from Lemmas \ref{lemA1} and \ref{lem-hk} that 
	\begin{equation}\label{thm2-p1}
		\|\nabla^j(\phi,u,v)(t)\|_{L^2}\leq C(1+t)^{-\frac{3}{4}-\frac{j}{2}},\quad 0\leq j\leq s,
	\end{equation}
	which, together with \eqref{new-var}, yields \eqref{Upper}. This completes the proof of Theorem \ref{thm1}.\endproof
	\vskip4mm
	\noindent{\it\textbf{Proof of Theorem \ref{thm3}.}\ }
Firstly, we have 
	\begin{align*}
		&\int_0^t \|G(t-\tau)*F(\tau)\|_{L^2}d\tau \\
		&\leq \int_0^t\Big(\|(G_{11}+G_{21})(t-\tau)*f_1(\tau)\|_{L^2} +\|(G_{12}+G_{22}+G_{32})(t-\tau)*f_2(\tau)\|_{L^2}\Big)d\tau\\
		&\quad+\int_0^t \|(G_{23}+G _{33})(t-\tau)*f_3(\tau)\|_{L^2}d\tau\\
		&\leq C\int_0^t(1+t-\tau)^{-\frac{3}{4}}(\|(u\cdot \nabla\phi)(\tau)\|_{L^1}+\|(u\cdot\nabla u)(\tau)\|_{L^1}+\|(v\cdot\nabla v)(\tau)\|_{L^1}+\|((e^\phi-1)(u-v))(\tau)\|_{L^1})d\tau\\
		&\quad+C\int_0^te^{-R(t-\tau)}(\|(u\cdot \nabla\phi)(\tau)\|_{L^2}+\|(u\cdot\nabla u)(\tau)\|_{L^2}+\|(v\cdot\nabla v)(\tau)\|_{L^2}+\|((e^\phi-1)(u-v))(\tau)\|_{L^2})d\tau.
	\end{align*}
	According to \eqref{lem-hs5}, we can obtain that
\begin{equation}\label{N001}
	\begin{aligned}
&\|(u\cdot \nabla\phi)(\tau)\|_{L^1}+\|(u\cdot\nabla u)(\tau)\|_{L^1}+\|(v\cdot\nabla v)(\tau)\|_{L^1}+\|((e^\phi-1)(u-v))(\tau)\|_{L^1}\\
		&\leq \|u\|_{L^2}\|\nabla\phi\|_{L^2}+\|u\|_{L^2}\|\nabla u\|_{L^2}
		+\|v\|_{L^2}\|\nabla v\|_{L^2}+C\|\phi\|_{L^2}\|u-v\|_{L^2}\\
		&\leq C(1+\tau)^{-\frac{3}{4}}(\varepsilon_0+ \mathcal{I}_0)(\|\nabla(\phi,u,v)\|_{L^2}+\|u-v\|_{L^2}),
	\end{aligned}
\end{equation}
and 
\begin{equation}\label{N002}
	\begin{aligned}
	&\|(u\cdot \nabla\phi)(\tau)\|_{L^2}+\|(u\cdot\nabla u)(\tau)\|_{L^2}+\|(v\cdot\nabla v)(\tau)\|_{L^2}+\|((e^\phi-1)(u-v))(\tau)\|_{L^2}\\
		&\leq \|u\|_{L^\infty}\|\nabla\phi\|_{L^2}+\|u\|_{L^\infty}\|\nabla u\|_{L^2}
		+\|v\|_{L^\infty}\|\nabla v\|_{L^2}+\|e^\phi-1\|_{L^\infty}\|u-v\|_{L^2}\\
		&\leq C\|\nabla u\|_{H^1}\|\nabla\phi\|_{L^2}+C\|\nabla u\|_{H^1}\|\nabla u\|_{L^2}
		+C\|\nabla v\|_{H^1}\|\nabla v\|_{L^2}+C\|\nabla (e^\phi-1)\|_{H^1}\|u-v\|_{L^2}\\
		&\leq C\|\nabla u\|_{H^1}\|\nabla\phi\|_{L^2}
		+C\|\nabla u\|_{H^1}\|\nabla u\|_{L^2}+C\|\nabla v\|_{H^1}\|\nabla v\|_{L^2}+C\|\nabla\phi\|_{H^1}\|u-v\|_{L^2}\\
		&\leq C(1+\tau)^{-\frac{5}{4}}(\varepsilon_0+ \mathcal{I}_0)(\|\nabla(\phi,u,v)\|_{L^2}+\|u-v\|_{L^2}).
	\end{aligned}
\end{equation}
	Consequently, it then follows from \eqref{N001}, \eqref{N002} and \eqref{D1} to prove 
	\begin{equation}\label{thm3-p2}
		\begin{aligned}
			&\int_0^t \|G(t-\tau)*F(\tau)\|_{L^2}d\tau\\
			&\leq C(\varepsilon_0+ \mathcal{I}_0)\int_0^t(1+t-\tau)^{-\frac{3}{4}}(1+\tau)^{-\frac{3}{4}}(\|\nabla (\phi,u,v)(\tau)\|_{L^2}+\|(u-v)(\tau)\|_{L^2})d\tau\\
			&\quad +C(\varepsilon_0+ \mathcal{I}_0)\int_0^te^{-R(t-\tau)}(1+\tau)^{-\frac{5}{4}}(\|\nabla (\phi,u,v)(\tau)\|_{L^2}+\|(u-v)(\tau)\|_{L^2})d\tau\\
			&\leq C(\varepsilon_0+ \mathcal{I}_0)\left(\int_0^t\big((1+t-\tau)^{-\frac{3}{2}}(1+\tau)^{-\frac{3}{2}}+e^{-2R(t-\tau)}(1+\tau)^{-\frac{5}{2}}\big)d\tau\right)^{\frac{1}{2}}\\
			&\quad\cdot \Big(\int_0^t(\|\nabla (\phi,u,v)(\tau)\|_{L^2}^2+\|(u-v)(\tau)\|_{L^2}^2)d\tau\Big)^{\frac{1}{2}}\\
			&\leq C(1+t)^{-\frac{3}{4}}(\varepsilon_0^2+\varepsilon_0 \mathcal{I}_0).
		\end{aligned}
	\end{equation}
	By Parseval's equality, Proposition \ref{pro-li-d}, and the smallness of $\varepsilon_0$, one has
	\begin{equation}\label{thm3-p1}
		\begin{aligned}
			\|U(t)\|_{L^2}
			&\geq \|G*U_0\|_{L^2} -\int_0^t \|G(t-\tau)*F(\tau)\|_{L^2}d\tau\\
	        &\geq c_*(1+t)^{-\frac{3}{4}}- C(1+t)^{-\frac{3}{4}}
		(\varepsilon_0^2+\varepsilon_0 \mathcal{I}_0)\\
		&\geq \frac{1}{2}c_*(1+t)^{-\frac{3}{4}}.
		\end{aligned}
	\end{equation}
	If $t$ is large enough, we can show 
	\begin{equation}\label{thm3-p4}
		\begin{aligned}
			\|\Lambda^{-1}U(t)\|_{L^2}&\leq \|\Lambda^{-1}U^\ell(t)\|_{L^2}
			+\|\Lambda^{-1}U^h(t)\|_{L^2}\\
			&\leq C(1+t)^{-\frac{1}{4}}+C\int_0^t(1+t-\tau)^{-\frac{1}{4}}(\|\nabla (\phi,u,v)(\tau)\|_{L^2}+\|(u-v)(\tau)\|_{L^2})d\tau+C\|U^h(t)\|_{L^2}\\
				&\leq C(1+t)^{-\frac{1}{4}}+C\int_0^t(1+t-\tau)^{-\frac{1}{4}}(1+\tau)^{-\frac{5}{4}}d\tau+C\|U(t)\|_{L^2}\\
			&\leq C(1+t)^{-\frac{1}{4}}.
		\end{aligned}
	\end{equation}
	It then follows from Lemma \ref{lema5} that 
	\begin{equation}\label{51}
		\|U\|_{L^2}\leq C\|\Lambda^{-1}U\|_{L^2}^{\frac{j}{j+1}}
		\|\nabla^j U\|_{L^2}^{\frac{1}{j+1}},
	\end{equation}
	which, together with \eqref{thm3-p1} and \eqref{thm3-p4}, yields
	\begin{equation}\label{thm3-p5}
		\|\nabla^j(\phi,u,v)(t)\|_{L^2}=\|\nabla^jU(t)\|_{L^2}\geq d_*(1+t)^{-\frac{3}{4}-\frac{j}{2}},
	\end{equation}
	where $d_*$ is a positive constant.
	Substituting \eqref{new-var} into \eqref{thm3-p5} implies
	\begin{equation}
		\|\nabla^j(a-a_*,u,v)(t)\|_{L^2}\geq d_*(1+t)^{-\frac{3}{4}-\frac{j}{2}},\quad 0\leq j\leq s.
	\end{equation}
This completes the proof of Theorem \ref{thm3}.
	\endproof
	
	\section*{Acknowledgements}
	\hspace{2em}
	The work of the first author was partially supported by National Key R\&D Program of China 2021YFA1000800, and National Natural Sciences Foundation of China 11688101. The work of the second author is partially supported by the project funded by China Postdoctoral Science Foundation 2023M733691. The work of the third is partially supported by National Natural Science Foundation of China 12271114, Natural Science Foundation of Fujian Province 2022J01304 and Program for Innovative Research Team in Science and Technology in Fujian Province University Quanzhou High-Level Talents Support Plan 2017ZT012.	The work of the fourth author is partially supported by  National Natural Science Foundation of China 12001033.


\begin{thebibliography}{10}
	
	\bibitem{MR3227296}
	H.-O. Bae, Y.-P. Choi, S.-Y. Ha, and M.-J. Kang.
	\newblock Global existence of strong solution for the
	{C}ucker-{S}male-{N}avier-{S}tokes system.
	\newblock {\em J. Differential Equations}, 257(6):2225--2255, 2014.
	
	\bibitem{MR2226800}
	C.~Baranger, L.~Boudin, P.-E. Jabin, and S.~Mancini.
	\newblock A modeling of biospray for the upper airways.
	\newblock In {\em C{EMRACS} 2004---mathematics and applications to biology and
		medicine}, volume~14 of {\em ESAIM Proc.}, pages 41--47. EDP Sci., Les Ulis,
	2005.
	
	\bibitem{MR2041452}
	L.~Boudin, L.~Desvillettes, and R.~Motte.
	\newblock A modeling of compressible droplets in a fluid.
	\newblock {\em Commun. Math. Sci.}, 1(4):657--669, 2003.
	
	\bibitem{MR709743}
	R.~Caflisch and G.~C. Papanicolaou.
	\newblock Dynamic theory of suspensions with {B}rownian effects.
	\newblock {\em SIAM J. Appl. Math.}, 43(4):885--906, 1983.
	
	\bibitem{MR3465376}
	J.~A. Carrillo, Y.-P. Choi, and T.~K. Karper.
	\newblock On the analysis of a coupled kinetic-fluid model with local alignment
	forces.
	\newblock {\em Ann. Inst. H. Poincar\'{e} C Anal. Non Lin\'{e}aire},
	33(2):273--307, 2016.
	
	\bibitem{MR3546341}
	Y.-P. Choi.
	\newblock Global classical solutions and large-time behavior of the two-phase
	fluid model.
	\newblock {\em SIAM J. Math. Anal.}, 48(5):3090--3122, 2016.
	
	\bibitem{MR4316126}
	Y.-P. Choi and J.~Jung.
	\newblock On the {C}auchy problem for the pressureless
	{E}uler-{N}avier-{S}tokes system in the whole space.
	\newblock {\em J. Math. Fluid Mech.}, 23(4):Paper No. 99, 16, 2021.
	
	\bibitem{Choi23}
	Y.-P. Choi and J.~Jung.
	\newblock Large time behavior of solutions to the pressureless
	euler-navier-stokes system in the whole space.
	\newblock {\em arXiv:2112.14449v1}, 2023.
	
	\bibitem{Choi-Jung2021}
	Y.-P. Choi and J.~Jung.
	\newblock On the dynamics of charged particles in an incompressible flow: from
	kinetic-fluid to fluid-fluid models.
	\newblock {\em Commun. Contemp. Math.}, 25(7):Paper No. 2250012, 78, 2023.
	
	\bibitem{Choi20232}
	Y.-P. Choi, J.~Jung, and J.~Kim.
	\newblock A revisit to the pressureless euler-navier-stokes system in the whole
	space and its optimal temporal decay.
	\newblock {\em arXiv:2112.14449v2}, 2023.
	
	\bibitem{MR3403400}
	Y.-P. Choi and B.~Kwon.
	\newblock Global well-posedness and large-time behavior for the inhomogeneous
	{V}lasov-{N}avier-{S}tokes equations.
	\newblock {\em Nonlinearity}, 28(9):3309--3336, 2015.
	
	\bibitem{MR2729436}
	T.~Goudon, L.~He, A.~Moussa, and P.~Zhang.
	\newblock The {N}avier-{S}tokes-{V}lasov-{F}okker-{P}lanck system near
	equilibrium.
	\newblock {\em SIAM J. Math. Anal.}, 42(5):2177--2202, 2010.
	
	\bibitem{MR2106333}
	T.~Goudon, P.-E. Jabin, and A.~Vasseur.
	\newblock Hydrodynamic limit for the {V}lasov-{N}avier-{S}tokes equations. {I}.
	{L}ight particles regime.
	\newblock {\em Indiana Univ. Math. J.}, 53(6):1495--1515, 2004.
	
	\bibitem{MR2106334}
	T.~Goudon, P.-E. Jabin, and A.~Vasseur.
	\newblock Hydrodynamic limit for the {V}lasov-{N}avier-{S}tokes equations.
	{II}. {F}ine particles regime.
	\newblock {\em Indiana Univ. Math. J.}, 53(6):1517--1536, 2004.
	
	\bibitem{MR4420295}
	D.~Han-Kwan.
	\newblock Large-time behavior of small-data solutions to the
	{V}lasov-{N}avier-{S}tokes system on the whole space.
	\newblock {\em Probab. Math. Phys.}, 3(1):35--67, 2022.
	
	\bibitem{MR4076066}
	D.~Han-Kwan, A.~Moussa, and I.~Moyano.
	\newblock Large time behavior of the {V}lasov-{N}avier-{S}tokes system on the
	torus.
	\newblock {\em Arch. Ration. Mech. Anal.}, 236(3):1273--1323, 2020.
	
	\bibitem{huang2023}
	F.~Huang, H.~Tang, and W.~Zou.
	\newblock Global well-posedness and optimal time decay rates of solutions to
	the pressureless euler-navier-stokes system.
	\newblock {\em arXiv:2307.11581}, 2023.
	
	\bibitem{MR1643525}
	T.-P. Liu and W.~Wang.
	\newblock The pointwise estimates of diffusion wave for the {N}avier-{S}tokes
	systems in odd multi-dimensions.
	\newblock {\em Comm. Math. Phys.}, 196(1):145--173, 1998.
	
	\bibitem{MR564670}
	A.~Matsumura and T.~Nishida.
	\newblock The initial value problem for the equations of motion of viscous and
	heat-conductive gases.
	\newblock {\em J. Math. Kyoto Univ.}, 20(1):67--104, 1980.
	
	\bibitem{MR109940}
	L.~Nirenberg.
	\newblock On elliptic partial differential equations.
	\newblock {\em Ann. Scuola Norm. Sup. Pisa Cl. Sci. (3)}, 13:115--162, 1959.
	
	\bibitem{1981Collective}
	P.~J. O'Rourke.
	\newblock Collective drop effects on vaporizing liquid sprays.
	\newblock {\em Phd Thesis Princeton University}, 1981.
	
	\bibitem{2006Large}
	I.~Vinkovic, C.~Aguirre, S.~Simo\'{l}Ns, and M.~Gorokhovski.
	\newblock Large eddy simulation of droplet dispersion for inhomogeneous
	turbulent wall flow.
	\newblock {\em International Journal of Multiphase Flow}, 32(3):344--364, 2006.
	
	\bibitem{MR2917409}
	Y.~Wang.
	\newblock Decay of the {N}avier-{S}tokes-{P}oisson equations.
	\newblock {\em J. Differential Equations}, 253(1):273--297, 2012.
	
	\bibitem{Williams1958Spray}
	Williams and A.~F.
	\newblock Spray combustion and atomization.
	\newblock {\em Physics of Fluids}, 1(6):541, 1958.
	
	\bibitem{MR4175837}
	G.~Wu, Y.~Zhang, and L.~Zhou.
	\newblock Optimal large-time behavior of the two-phase fluid model in the whole
	space.
	\newblock {\em SIAM J. Math. Anal.}, 52(6):5748--5774, 2020.
	
	\bibitem{Wuzhou2021}
	Z.~Wu and W.~Zhou.
	\newblock Pointwise space-time estimates of two-phase fluid model in dimension
	three.
	\newblock {\em arXiv:2111.01987}, 2021.
	
	\bibitem{MR3073216}
	C.~Yu.
	\newblock Global weak solutions to the incompressible
	{N}avier-{S}tokes-{V}lasov equations.
	\newblock {\em J. Math. Pures Appl. (9)}, 100(2):275--293, 2013.
	
\end{thebibliography}
\end{document}